\documentclass[english]{article}
\usepackage{preamble}
\usepackage{macros}
\usepackage{quiver}

\begin{document}

\title{Nilpotence phenomena over \texorpdfstring{$\mathbb{E}_\infty$}{E∞}-rings in positive characteristic}
\author{Florian Riedel}
    \date{}
\maketitle
\begin{abstract}
We study the behavior of nilpotent classes in $\mathbb{E}_\infty$-rings over finite fields. In characteristic $p=2$, we show that any nilpotent class can be killed by a nilpotence detecting map of $\mathbb{E}_\infty$-rings. As a consequence, we prove the existence of ``residue fields'', namely that any non-zero $\mathbb{E}_\infty$-ring over $\mathbb{F}_2$ admits a map to a separably closed graded field. Conversely, for any odd prime $p$, we construct non-zero $\mathbb{E}_\infty$-rings over $\mathbb{F}_p$ which admit nilpotent classes that cannot be killed by any non-zero map of $\mathbb{E}_\infty$-rings, a phenomenon that does not occur classically or at any finite height.
\end{abstract}
\vspace{1cm}
\begin{figure}[H]
  \centering{}
  \includegraphics[scale=0.26]{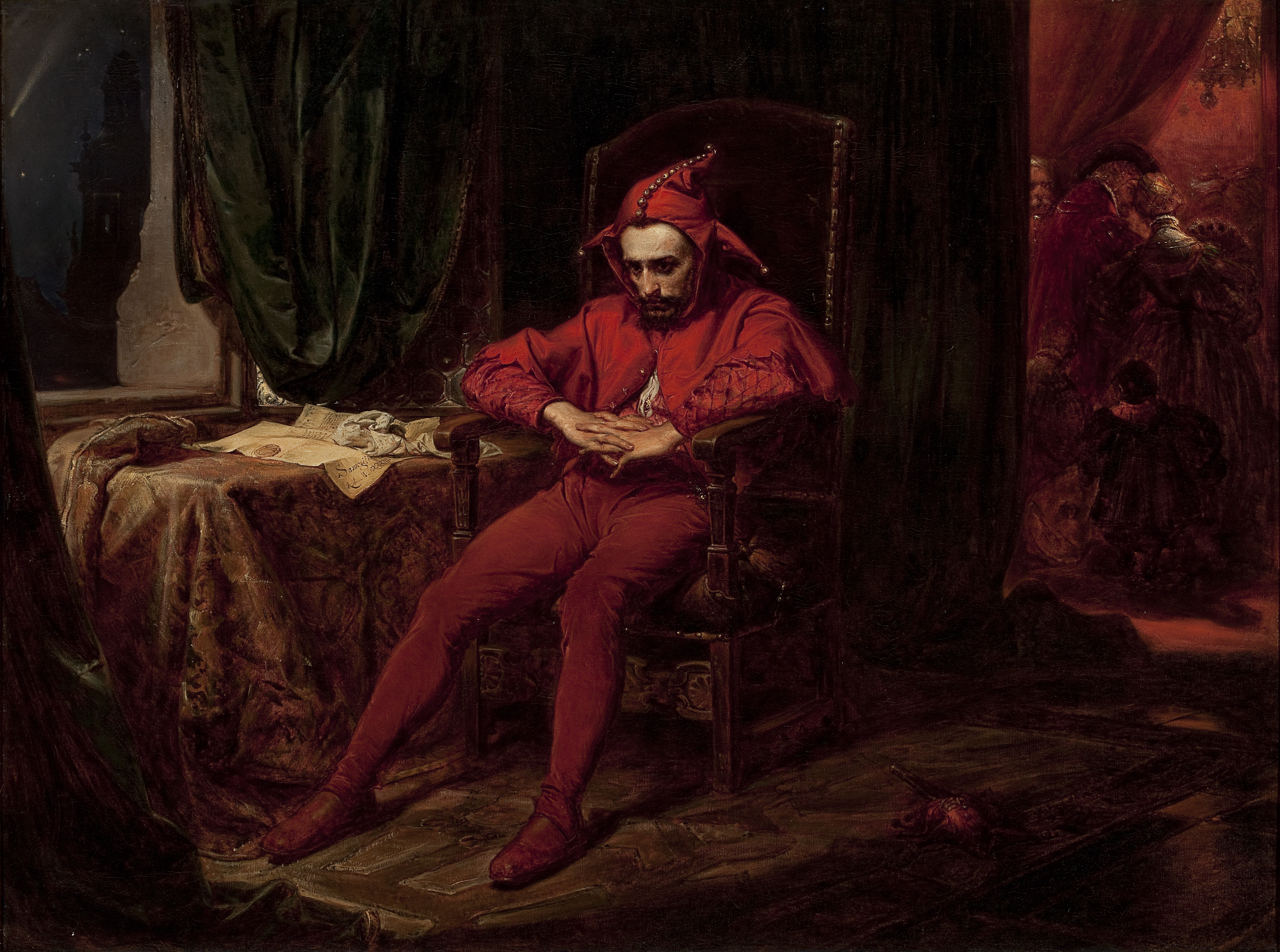}
  \caption*{\textit{Sta\`nczyk} by Jan Matejko}
\end{figure}
\newpage

\tableofcontents{}

\section{Introduction}\label{sec: intro}
\subsection*{Overview}

The role of nilpotent classes in ordinary algebraic geometry is well understood: They detect infinitesimal information, but no topological data. More precisely, if $R\in \CAlg(\mathrm{Ab})$ is a commutative
ring with nilradical $N\subseteq R$ then the induced map $\mathrm{Spec}(R/N)\to \mathrm{Spec}(R)$
is an equivalence of topological spaces. 

This changes drastically when working with $\E_\infty$-ring spectra. 
By Hurewicz combined with the Nishida nilpotence theorem~\cite{nishidanil} we know that the
graded commutative ring $\pi_*\SS$ has $\pi_0$ given by $\ZZ$ and only 
nilpotent classes in higher degrees. Nonetheless, the sphere spectrum has a
geometry attached to it that is much richer than that of $\mathrm{Spec}(\ZZ)$\footnote{Namely the moduli stack of formal groups.}, controlled by the map 
$\SS\to \mathrm{MU}$ to the complex cobordism spectrum.
This leads us to ask generally:

\begin{question}\label{ques: reduced geometry}
    Given some arbitrary $\E_\infty$-ring $A$, how much of the geometry of 
    $A$ is seen by $\E_\infty$-ring maps $A\to T$ such that the ring $\pi_*T$ is
    reduced?
\end{question}

Indeed, in the present work we will be modest and study whether \emph{anything at all} about
$A$ is detected by such maps, i.e.~whether we can always find a map
to a non-zero reduced ring.

If $A$ is connective and non-zero this is clear\footnote{By considering the truncation map $A\to \tau_{\leq 0}A$ and taking
the quotient by the nilradical of the ordinary commutative ring $\pi_0 A$.}.
For non-connective rings the first major result in this direction
is the Chromatic Nullstellensatz of Burklund--Schlank--Yuan. If for any prime $p$ and any $0\leq n <\infty$ the telescopic localization 
$A\to L_{T(n)} A$ is non-zero, then~\cite[\href{https://arxiv.org/pdf/2207.09929\#theorem.1.4}{Theorem D}]{cnst} says that
the ring $A$ admits a map $A\to E_n$ where $E_n$ is a height $n$ Lubin--Tate theory, which in particular has reduced homotopy groups.

We continue this investigation by considering the \enquote{infinite height} case of $\E_\infty$-rings
over the finite field $\F_p$. Here the situation turns out to differ drastically depending
on whether $p$ is odd or even. For $p=2$ our main result is the following:

\begin{thmx}[\cref{thm: existence of reduced points}]\label{thm: thmA}
   Any non-zero $\E_\infty$-ring $A$ over $\F_2$ admits a map of $\E_\infty$-rings $A\to T$ such that $\pi_0T$ is a separably closed field
  and $\pi_{\ast}T\simeq \pi_0T[u^{\pm}]$ for some class $u\in \pi_1T$. 
\end{thmx}

Conversely, for any odd prime $p$ we construct a non-zero
$\E_\infty$-ring which contains an ``immortal'' nilpotent class
that cannot be killed by any map to a non-zero $\E_\infty$-ring. More precisely, let $\F_p\{x\}$
denote the free $\E_\infty$-algebra over $\F_p$ on a class
in topological degree 0. We show the following:

\begin{thmx}[\cref{thm: odd prime}]\label{thm: thmB}
For any odd prime $p$, there exists a class $\theta(x)\in \pi_{\ast}\F_p\{x\}$, such that the 
image of $\theta(x)$ under the $\E_\infty$-cofiber $\F_p\{x\}\to\modEn{\F_p\{x\}}{\infty}{x^p}$
is not nilpotent. In particular, the localization 
$(\modEn{\F_p\{x\}}{\infty}{x^p})[\theta(x)^{-1}]$ is a non-zero $\E_\infty$-ring which admits no $\E_\infty$-maps out to a non-zero ring with reduced homotopy groups.
\end{thmx}

Indeed, any $\E_{\infty}$-map to a ring $T$ such that $\pi_0T$ is reduced takes the nilpotent
class $x$ to zero.
In particular, it must factor through the $\E_\infty$-cofiber by $x$, which kills the
invertible class $\theta(x)$, forcing $T\simeq 0$.
In fact, as explained in \cref{rem: E2 reduced point}, the class $\theta(x)$ is already
killed by the $\E_2$-cofiber, and so the same conclusion holds if we only ask for an $\E_2$-algebra map.

Approaching \cref{ques: reduced geometry} from the other end, we study those
$\E_\infty$-rings over $\F_p$ which one can always universally map out to.
More precisely, as a consequence of~\cite[\href{https://arxiv.org/pdf/2207.09929\#nul.A.3}{Theorem A.3}]{cnst}
\emph{any} non-zero $\E_\infty$-ring $A$ admits an $\E_\infty$-map $A\to T$ such that 
$T\in \CAlg_A$ is \tdef{nullstellensatzian}\footnote{Meaning every compact algebra 
$B\in \CAlg_{T}^{\omega}$ admits a $T$-algebra map $B\to T$.}. 
We are able to extract the following descriptions:

\begin{proposition}[\cref{prop: even prime nstz description}, \cref{prop: odd prime nstz description}]\label{prop: nstz intro}
For any nullstellensatzian $T\in \CAlg_{\F_p}$ the following
hold:
\begin{enumerate}
    \item  If $p=2$ then $\pi_{*}T$
is a 1-periodic, separably closed field which is not perfect.
  \item If $p>2$ the graded ring $\pi_{\ast}T$
  has the following properties:
  \begin{enumerate}
      \item $\pi_{\ast}T$ is 2-periodic and $\pi_0 T$ is a local ring of Krull dimension zero.
      \item $\pi_1T \neq 0$
      \item Any nilpotent class $v\in \pi_0 T$ satisfies $v^p=0$.
      \item The Nilradical $N\subseteq \pi_0 T$ is not a nilpotent ideal. In particular,
            $\pi_0T$ is neither reduced nor Noetherian.
  \end{enumerate}
\end{enumerate}
\end{proposition}

Using these descriptions, we also deduce that the nullstellensatzian objects
fail to satisfy a nilpotence theorem in $\CAlg_{\F_p}$.
More precisely, say that a map of $\E_\infty$-rings
$f\colon A\to B$ \tdef{detects nilpotence} if, for any $A$-linear map $v\colon \sI \to A$
with $\sI\in \Mod_A^{\omega}$, the composite $f(v)$ being null implies that 
$v^{\otimes n}\simeq 0$ for some $n>0$.

\begin{corollary}[\cref{cor: even prime nstzians do not detect nilp}, \cref{cor: odd prime nstzians do not detect nilp}]
    For any prime $p$, the family of $\E_\infty$-maps $\F_p\{x\}\to T$ such that
    $T$ is nullstellensatzian \emph{does not} jointly detect nilpotence.
\end{corollary}

We do not presently know what to make of these results, other than to 
remark that this strong divergence from ordinary commutative algebra
and higher algebra in the $T(n)$-local category is quite strange.
It appears that genuinely new insights will be required to gain a good
understanding of the local theory of $\E_\infty$-algebras over $\F_p$. 
 
\subsection*{Strategy and methods}

The key to proving \cref{thm: thmA} and \cref{thm: thmB} is to control what happens when we
take the $\E_\infty$-cofiber by a nilpotent class.

Let us again begin with discussing the $p=2$ case. Given a non-zero $A\in \CAlg_{\F_2}$ and a nilpotent
class $v\in \pi_{\ast} A$, we want to get rid of $v$ by taking the $\E_\infty$-cofiber
$A\to \modEn{A}{\infty}{v}$, while ensuring that this does not accidentally produce the zero ring.
We do this by showing something stronger, namely that
nilpotent classes die in a nilpotence detecting way. The following proposition is the main technical
result over $\F_2$.

\begin{proposition}[\cref{prop: killing nilpotence}]\label{prop: killing nilp intro}
  For any $A\in \CAlg_{\F_2}$ and any nilpotent class $v\in \pi_{\ast}A$, the
 $\E_\infty$-cofiber $A\to \modEn{A}{\infty}{v}$ detects nilpotence.
\end{proposition}

The straightforward reason the map $A\to \modEn{A}{\infty}{v}$ might \emph{fail} to
detect nilpotence is that it not only kills the class $v$, but also the extended powers $Q_i(v)$
and their iterates. While a simple computation with the Cartan formula, see \cref{lem: Qnilpotent},
shows that these classes are all again nilpotent, we need to quantify the idea
that these classes dying is, in some sense, the only thing that happens in the cofiber.

By work of Araki--Kudo~\cite{ArakiKudo} and Browder~\cite{browder}, we know that
for any $0\leq n\leq \infty$ the free $\E_n$-algebra on a generator $\F_2\{x\}_n$ has
homotopy groups given by an infinite polynomial ring on the (iterated) extended powers
$Q_{1}^{i_1}\cdots Q_{n-1}^{i_{n-1}}(x)$.
By constructing a highly structured comparison map, we deduce that, for any
$k\leq n$, the underlying $\E_k$-algebra of $\F_2\{x\}_n$ is equivalent to an infinite
tensor product of free $\E_k$-algebras, cf.~\cref{prop: En decomposition F2}. From
this, we obtain the following decomposition of the filtered $\E_n$-cofiber by a class
in filtration degree one.

\begin{lemma}[\cref{cor: filtered cofiber p=2}]
  Let $A\in \EAlg{n+1}{\Mod_{\F_2}}$ for some $n\geq 0$ and let $v\in \pi_\ast A$ be some class.
  For any $k\geq n$, we have an equivalence of filtered $\E_{k-1}$-algebras
  \[ \bigotimes_{(i_{k+1},\dots, i_n)} \modEn{A}{k-1}{Q^{i_{k+1}}_{k+1}\cdots Q^{i_{n}}_n(\tau v)}
    \xrightarrow{\sim}\modEn{A}{n}{\tau v}\,,\]
  where the infinite tensor product runs over all sequences of non-negative integers $(i_{k+1},\dots, i_n)$.
\end{lemma}

In particular, for $k=n=1$, we learn that the $\E_1$-cofiber $\modEn{A}{1}{v}$
of a class in an $\E_2$-algebra $A$ has a natural multiplicative filtration whose $(2^{i+1}-1)$ stage is given by the $A$-module
$A/(v,Q_1(v),\dots, Q_1^i(v))$ cf.~\cref{lem: fil E1 cofib}. This analysis gives
us enough control to deduce \cref{prop: killing nilp intro} from the fact that
$v$ being nilpotent forces the extended powers to be nilpotent, and \cref{thm: thmA}
is a straightforward consequence.

For the case of an odd prime $p>2$, we want to employ the same methods to arrive
at the opposite conclusion. The name of the game is again to understand the free algebra and the $\E_\infty$-cofiber.

By work of Dyer--Lashof~\cite{dyerlash} and Cohen~\cite{cohen}, we know that the homotopy
groups of the free $\E_\infty$-algebra $\F_p\{x\}$ form a free \emph{graded commutative}\footnote{Meaning in particular that classes
in odd topological degree square to zero.} $\F_p$-algebra on certain composites of the
extended $p$-th powers $P_i(x),\beta P_i(x)$\footnote{See \cref{cons: odd operations} for a precise definition and our indexing conventions.}. 
In particular, any even degree class is not nilpotent in $\pi_{\ast}\F_p\{x\}$.
Moreover, in \cref{prop: bosonic nilpotent} we observe that $P_I(x^p)\simeq 0$ in $\F_p\{x\}$ 
for any iterated operation $P_I$ containing at least one Bockstein. Further analyzing
the Cartan formula, we learn that the $p$-th power map
\[ \F_p\{x\}\too \F_p \{x\}, \quad x \mapsto x^p\]
does not hit any iterated operation $P_I(x)$ in even degree containing a Bockstein and 
thus the class
\[ \mdef{\theta(x)}\coloneq \beta P_{\frac{1}{2}}\beta P_1(x)\in \pi_{2(p^2+p-1)}\F_p\{x\}\]
has no a priori reason to become nilpotent in the $\E_\infty$-cofiber by $x^p$.

To ensure that nothing unexpected happens in the Tor spectral sequence, we first
use the computations of \cite{dyerlash} to deduce that, the underlying 
$\E_1$-algebra of $\F_p\{x\}$ is naturally equivalent to an infinite tensor product
of free $\E_1$-algebras on even degree classes and free $\E_2$-algebras on odd degree classes,
see \cref{prop: En decomposition Fp}. From this, we get the following:

\begin{lemma}[\cref{cor: unit decomp Fp}]
Let $A\in \CAlg_{\F_2}$ and let $v\in \pi_{\ast} A$ be a class in even topological
degree. Then the $\E_{\infty}$-cofiber $\modEn{A}{\infty}{v}$ is, as an $\E_0$-algebra over $A$,
equivalent to an infinite tensor product  
\[ \bigotimes_{I} A/P_I(v) \otimes \bigotimes_{J}\modEn{A}{1}{P_J(v)} 
\xrightarrow{\sim} \modEn{A}{\infty}{v}\]
of $\E_0$- and $\E_1$-cofibers, where the $P_I(v)$ sit in even and the $P_J(v)$
in odd topological degree.
\end{lemma}

Applying this to the ring $A=\F_p\{x\}$ and $v=x^p$, we see that most 
of the terms appearing are taking a cofiber by a class that is already null,
making the $A$-module structure very simple. 
Moreover, we have enough structure
to conclude that the class $\theta(x)$ does not act nilpotently on the cofiber
$\modEn{\F_p\{x\}}{\infty}{x^p}$, which is precisely the content of \cref{thm: thmB}.

\subsection*{Outline}

We begin in \cref{sec: graded} by reviewing some facts about locally graded categories
and the shearing construction. Moreover, we discuss free $\E_n$-algebras an $\E_{n-1}$-cofibers.
In \cref{sec: powerop}, we recall the construction of Dyer--Lashof operations in arbitrary
$\F_p$-linear, $\E_n$-monoidal categories and deduce the decompositions
of free algebras, as well as of the filtered $\E_{n-1}$-cofiber by a
class in filtration degree one.
Lastly in \cref{sec: points}, we first review some material on nullstellensatzian
rings and detecting nilpotence, before moving on to prove \cref{thm: thmA}
and \cref{thm: thmB}.

\subsection*{Conventions}
\begin{enumerate}
  \item We freely use the theory of $(\infty,1)$-categories as developed
  in~\cite{htt} and~\cite{ha} and henceforth refer to $(\infty,1)$-categories simply 
  as categories.
  \item We denote by $\Spc$ the category of spaces and by $\Sp$ the category of spectra.
  \item We write $\Prl$ for the category of presentable categories equipped with
        the Lurie tensor product. 
  \item By commutative ring we will always mean $\E_{\infty}$-ring and call the objects
        of $\CAlg(\rm{Ab})$ discrete commutative rings.
   \item If $A$ is an $\E_n$-algebra for some $1\leq n \leq \infty$, we write $\Mod_A$ for
         the category of left $A$-modules.
\end{enumerate}

\subsection*{Acknowledgments}
I would like to thank Robert Burklund and Maxime Ramzi for suggesting this project. I am especially grateful to Robert Burklund for his patience and the many, many
helpful discussions we had.
I would also like to thank Vignesh Subramanian for his insights into power operations and his hospitality at the Max Planck institute in Bonn,
as well as Anish Chedalavada for pointing out the failure of 
the Nullstellensatzians to detect nilpotence.
Finally, I am grateful to Achim Krause and Maxime Ramzi for comments on and spotting some
mistakes in earlier drafts.
The author was supported by the Danish National Research Foundation through the Copenhagen Center for
Geometry and Topology (DNRF 151) during the writing and conception of this paper, as well as by Germany's Excellence Strategy
EXC 2044 390685587, Mathematics M\"{u}nster: Dynamics-Geometry-Structure
during later revisions.


\section{Graded \texorpdfstring{$\E_{n}$}{En}-algebras and shearing}\label{sec: graded}
The main purpose of this section is to alleviate the authors own confusion about 
$\E_n$-algebras and locally graded categories, and as such is largely expository.

Firstly, in \cref{subsec: cofib loop} we review some well known
facts about $\E_n$-algebras which we extract from \cite{ha}. 
Concretely, we need the fact that $\E_{n-1}$-cofibers can be computed as a Bar
construction, which is \cref{lem: En relative bar}, see also \cite{Encofib} for a comprehensive
treatment. Moreover, we recall the description of the free $\E_n$-algebra 
as the homology of iterated loop spaces in \cref{lem: free algebra loop space}.

Secondly, we discuss graded and filtered categories, following the ideas
laid out in~\cite{rotation}. Moreover, we discuss how strict Picard elements
lead to a shearing equivalence \cref{lem: shearing}, which substantially
simplifies the computations of the free algebras in \cref{sec: powerop} and allows
us to construct the Dyer--Lashof operations efficiently via \cref{cons: shearing}. 

\subsection{\texorpdfstring{$\E_n$}{En}-cofibers and iterated loop spaces}\label{subsec: cofib loop}

Let $\cC$ be a presentably $\E_n$-monoidal category for some $0\leq n \leq \infty$.
By~\cite[\href{https://www.math.ias.edu/~lurie/papers/HA.pdf\#theorem.3.2.4.3}{3.2.4.3}]{ha}
combined with the Dunn-additivity
theorem~\cite[\href{https://www.math.ias.edu/~lurie/papers/HA.pdf\#theorem.5.1.2.2}{5.1.2.2}]{ha},
we know that the category $\EAlg{k}{\cC}$ inherits a natural $\E_{n-k}$-monoidal structure,
together with an $\E_{n-k}$-monoidal enhancement of the forgetful functor
$U_{\cC}\colon \EAlg{k}{\cC} \to \cC$.
By~\cite[\href{https://www.math.ias.edu/~lurie/papers/HA.pdf\#theorem.3.1.3.5}{3.1.3.5}]{ha},
the functor $U_{\cC}$ admits a left adjoint
\[ \mdef{\free{\1_{\cC}}{k}{-}}\colon \cC\too \EAlg{k}{\cC}\,.\]
For any $\sI\in \cC$ we call $\1_{\cC}\{\sI\}_k$ the \tdef{free $\E_k$-algebra} on $\sI$.
If $\sI=\Sigma^t\1_{\cC}$ is suspension of the unit for some integer $t$, we also
refer to $\1_{\cC}\{\Sigma^{t}\1_{\cC}\}$ as the free algebra on a generator in
topological degree $t$.

By~\cite[\href{https://www.math.ias.edu/~lurie/papers/HA.pdf\#theorem.3.4.4.6}{3.4.4.6}]{ha},
we know that for any $A\in \EAlg{n}{\cC}$ the category of (left)-modules $\Mod_A(\cC)$ naturally
inherits the structure of an $\E_{n-1}$-monoidal category via the Bar construction
 \[ M\otimes_A N = \colim ( M \otimes A^{\otimes \bullet} \otimes N), \quad M,N\in \cC\]
such that the colimit preserving base change
 \[ -\otimes A \colon \cC \too \Mod_{A}(\cC)\]
 naturally refines to a $\E_{n-1}$-monoidal functor. In particular,
 it induces a colimit preserving functor $\EAlg{n-1}{\cC}\to \EAlg{n-1}{\Mod_A(\cC)}$
 which we will use implicitly.

\begin{definition}
Let $\cC$ be a stable, presentably $\E_{n}$-monoidal category and let $A\in \EAlg{n}{\cC}$.
Given a map $v\colon \sI\to A$ in $\cC$, we denote by
 \[ A \too  \mdef{\modEn{A}{n-1}{v}} \in \EAlg{n-1}{\Mod_A(\cC)}.\]
the \tdef{$\E_{n-1}$-cofiber} of $v$, i.e.~the map of $\E_{n-1}$-algebras obtained by the pushout
\[\begin{tikzcd}
	{A\otimes \1_{\cC}\{\sI\}_{n-1}} & A \\
	A & \modEn{A}{n-1}{v}\,.
	\arrow["A\otimes \1_{\cC}\{v\}_{n-1}", from=1-1, to=1-2]
	\arrow["A\otimes \1_{\cC}\{0\}_{n-1}"', from=1-1, to=2-1]
	\arrow[from=1-2, to=2-2]
	\arrow[from=2-1, to=2-2]
\end{tikzcd}\]
in the category $\EAlg{n-1}{\Mod_{A}(\cC)}$.
\end{definition}

These cofibers admit a more explicit presentation using the Bar construction.

\begin{lemma}\label{lem: En relative bar}
Let $1\leq n \leq \infty$ and let $\cC$ be a presentably $\E_{n}$-monoidal category.
Suppose we are given $A\in \EAlg{n}{\cC}$ and a map $v\colon \sI\to A \in \cC$.
Then, the $\E_{n-1}$-cofiber by $v$ is computed by the Bar construction
\[ \modEn{A}{n-1}{v}\xrightarrow{\sim} A \otimes_{\free{\1_{\cC}}{n}{\sI}}\1_{\cC}
\in \EAlg{n-1}{\Mod_A}\,.\]
\end{lemma}
\begin{proof}
The case $n=\infty$ is an easy consequence
of~\cite[\href{https://www.math.ias.edu/~lurie/papers/HA.pdf\#theorem.3.2.4.7}{3.2.4.7}]{ha}. Moreover, for finite $n$, we know by~\cite[\href{https://www.math.ias.edu/~lurie/papers/HA.pdf\#theorem.5.3.1.16}{5.3.1.16}]{ha} 
that the $\E_1$-monoidal structure on $\EAlg{n-1}{\cC}$ preserves colimits of
contractible shape in each variable. 
Consequently, by~\cite[\href{https://www.math.ias.edu/~lurie/papers/HA.pdf\#theorem.5.2.2.12}{5.2.2.12}]{ha},
    the unit $\1_{\cC}$ considered as an object of $\EAlg{n-1}{\Mod_{\1\{\sI\}_{n}}}$
    via the 0-augmentation is given by the pushout
\[\begin{tikzcd}
	{\1_\cC\{\sI\}_{n-1} \otimes \1_{\cC}\{\sI\}_{n}} & {\1_\cC\{\sI\}_{n}} \\
	{\1_\cC\{\sI\}_{n}} & {\1_{\cC}}\,.
	\arrow["m", from=1-1, to=1-2]
	\arrow["e"', from=1-1, to=2-1]
	\arrow[from=1-2, to=2-2]
	\arrow[from=2-1, to=2-2]
\end{tikzcd}\]
Here, $e$ is induced by the 0 augmentation of $\1_{\cC}\{\sI\}_{n-1}$ and $m$ is induced by the
natural map $\1_{\cC}\{\sI\}_{n-1} \to \1_{\cC}\{\sI\}_{n}$.
Base changing along the map of $\E_n$-algebra  $\1_{\cC}\{v\}\colon \1_{\cC}\{\sI\}_{n}\to A$ and using that the base change preserves colimits, we learn that we have an equivalence
    \begin{align*}
        A\otimes_{\free{\1_{\cC}}{n}{\sI}} \1_{\cC} &\simeq A
        \otimes_{\free{\1_{\cC}}{n}{\sI}} (\1_{\cC}\{\sI\}_{n} \amalg_{\1_{\cC}\{\sI\}_{n}\otimes \1_{\cC}\{\sI\}_{n-1}} \1_{\cC}\{\sI\}_{n})\\
        &\simeq  A \amalg_{A \otimes \free{\1_{\cC}}{n-1}{\sI}} A
    \end{align*}
    of $\E_{n-1}$-algebra in $\Mod_{A}$.
\end{proof}

\begin{remark}\label{rmk: En cofiber operation}
  Let $\cC$ be a stable, presentably $\E_{n}$-monoidal category and $Q\colon \sJ \to \1_{\cC}\{\sI\}_{n}$
  be a class in the free $\E_n$-algebra. Then for any $A\in \EAlg{n}{\cC}$ and any class
  $v \colon \sI \to A$, the presentation of the $\E_{n-1}$-cofiber in \cref{lem: En relative bar}
  provides us with a nullhomotopy of the composite map
  \[ \sJ \xrightarrow{Q(v)} A \too \modEn{A}{n-1}{v}\,.\]
  In particular, this tells us that any $\E_{n-1}$-algebra map $A\to B$ that kills $v$
  also kills the class $Q(v)$, even though the operation $Q$ may not be defined on
  the $\E_{n-1}$-algebra $B$.
\end{remark}

\begin{construction}\label{cons: general decomp}
  Let $\cC$ be a presentably $\E_{n+1}$-monoidal category.
  Given two augmented $\E_{n}$-algebras
  $A,B\in \EAlg{n}{\cC}_{/\1_{\cC}}$ we can form their tensor product to obtain
  an augmented $\E_{n}$-algebra $C=A\otimes B$. Using the projection formula,
  we obtain an equivalence
  \begin{align*}
   (\1_{\cC} \otimes_{A} C ) \otimes_C (C \otimes_{B} \1_{\cC})
  &\simeq (C\otimes_A \1_{\cC}) \otimes_{B} \1_{\cC}\\
  &\simeq ((A\otimes B) \otimes_A \1_{\cC}) \otimes_{B} \1_{\cC}\\
  &\simeq B \otimes_{B} \1_{\cC}\\
  &\simeq \1_{\cC}\,,
  \end{align*}
  of $\E_{n-1}$-algebras in $\Mod_{C}(\cC)$. We call this the \tdef{unit decomposition} associated
  to the decomposition of augmented algebras $C\simeq A\otimes B$.
\end{construction}

In \cref{sec: powerop} we will utilize \cref{cons: general decomp} combined with 
\cref{lem: En relative bar} to deduce from decompositions of free $\E_n$-algebras,
associated decompositions of $\E_{n-1}$-cofibers. The main tool for computing
free $\E_n$-algebras in the first place is the following story.

The free $\E_k$-algebra $\free{\1_{\cC}}{k}{\Sigma^t\1_{\cC}}$ on a generator
in positive topological degree $t$ admits a well known presentation in
terms of the homology of loop spaces, which we want to record here.
Since $\cC$ is presentably $\E_{n}$-monoidal, it comes equipped with a $\E_{n}$-monoidal
functor
\[ \1_{\cC}[-]\colon \Spc \too \cC\,,\]
which takes a space $X\in \Spc$ to the constant colimit $\1_{\cC}[X]= \colim_{X}\1_{\cC}$,
which we think of the homology of $X$ with coefficients in $\cC$.
If $\cC$ is moreover stable, then for any pointed space $(y\colon \pt \to Y)\in \Spc_{\ast}$ write
\[ \1_{\cC}^{\rm{red}}[Y]\coloneq \cofib(\1_{\cC} \xrightarrow{\1_{\cC}[y]} \1_{\cC}[Y])\]
for the reduced homology of $Y$.

This allows for the following description of the free $\E_n$-algebra on a reduced
homology object.

\begin{proposition}\label{lem: free algebra loop space}
  Let $0\leq n \leq \infty$ and let $\cC\in \EAlg{n}{\Prlst}$.
  Then for any pointed, connected space $X$, we have a natural equivalence
  \[ \free{\1_{\cC}}{n}{\1_{\cC}^{\rm{red}}[X]} \simeq \1_{\cC}[\Omega^{n}\Sigma^{n}X] \in \EAlg{k}{\cC}\]
  of $\E_{n}$-algebras in $\cC$. In particular, taking $X= S^{t}$ for any $t\geq 1$, we get
  an equivalence
\[ \free{\1_{\cC}}{n}{\Sigma^{t}\1_{\cC}} \simeq \1_{\cC}[\Omega^{n}S^{n+t}] \in \EAlg{n}{\cC}.\]
\end{proposition}
\begin{proof}
  Considering the commutative diagram in $\Pr^{\rm{R}}$ given by the forgetful functors
  \[\begin{tikzcd}
	{\EAlg{n}{\Spc}} & {\EAlg{n}{\cC}} \\
	{\Spc_\pt} & \cC\,,
	\arrow["{\Omega^\infty}"', from=1-2, to=1-1]
	\arrow["U_{\Spc_{\ast}}"', from=1-1, to=2-1]
	\arrow["U_{\cC}", from=1-2, to=2-2]
	\arrow["{\Omega^\infty(0\to (-))}", from=2-2, to=2-1]
\end{tikzcd}\]
we may pass to the left adjoints to obtain a diagram in $\Prl$ of the form
\[\begin{tikzcd}
	{\EAlg{n}{\Spc}} & {\EAlg{n}{\cC}} \\
	{\Spc_\pt} & \cC\,.
	\arrow["{\1[-]}", from=1-1, to=1-2]
	\arrow["{F^{\rm{red}}_k(-)}", from=2-1, to=1-1]
	\arrow["{\1^{\rm{red}}[-]}"', from=2-1, to=2-2]
	\arrow["{\1\{-\}_k}"', from=2-2, to=1-2]
\end{tikzcd}\]
It thus suffices to show that, for any connected $X\in \Spc_{\pt}$,
we have $F_{k}^{\rm{red}}(X)\simeq \Omega^{n}\Sigma^{n}X$. For $n< \infty$ this is an easy
consequence of the recognition theorem~\cite[\href{https://www.math.ias.edu/~lurie/papers/HA.pdf\#theorem.5.2.6.10}{5.2.6.10}]{ha} combined with
Dunn-additivity~\cite[\href{https://www.math.ias.edu/~lurie/papers/HA.pdf\#theorem.5.1.2.2}{5.1.2.2}]{ha}. The $n=\infty$ case then follows readily by taking the filtered colimit
and using~\cite[\href{https://www.math.ias.edu/~lurie/papers/HA.pdf\#theorem.5.1.1.5}{5.1.1.5}]{ha}.
\end{proof}

\subsection{Recollections on locally graded categories}

Let $\cV\in \CAlg(\Prl)$ be a presentably symmetric monoidal category and write
$\Prl_{\cV}=\Mod_{\cV}(\Prl)$ for the category of $\cV$-linear categories.
Throughout this section, we fix some $\cV \in \CAlg(\Prl)$ unless otherwise specified.

The unique, colimit preserving, symmetric monoidal functor $\Spc \to \Prl_{\cV}$ taking
the point to $\cV$ induces a colimit preserving functor
 \[ \mdef{\cV[-]}\colon \Spcn \too \CAlg(\Prl_{\cV}) \,,\]
 whose right adjoint takes a $\cV$-linear, presentably symmetric monoidal category 
$\cC$ to the connective spectrum $\mdef{\pic(\cC)}\in \Spcn$ of $\otimes$-invertible elements 
in $\cC^{\simeq}$ called the \tdef{Picard spectrum}.

    \begin{definition}
      For any $\cC\in \Prl_{\cV}$ we call
      \[ \mdef{\cC^{\gr}}\coloneq \cC \otimes_{\cV} \cV[\ZZ]\,\]
      the category of \tdef{$\ZZ$-graded objects} in $\cC$. 
      Moreover, we refer to $\mdef{\Prl_{\cV^{\gr}}}$ as the category of     \tdef{locally $\ZZ$-graded}, $\cV$-linear categories.
    \end{definition}
    Since we only talk about $\ZZ$-gradings, we often leave the $\ZZ$ implicit
    and refer only to graded objects and locally graded categories.

\begin{remark}
Let us make some observations about categories of graded objects.
\begin{enumerate}
    \item The symmetric monoidal category $\cV^{\gr}$ is equivalent to the functor 
          category $\Fun(\ZZ,\cV)$ carrying the Day-convolution monoidal structure. 
          Thus, we think of an object of $\sI_{\ast}\in \cV^{\gr}$ as an integer indexed 
          sequence $\{\sI_n\}_{n\in \ZZ}$ of objects $\sI_n \in \cV$.
    \item The inclusion $0\to \Z$ induces a fully faithful, symmetric monoidal functor 
          \[\cV\too \cV^{\gr},\quad \sI \mapsto \sI(0)\]
          which takes an object $\sI\in \cV$ to the graded object 
          $\sI(0)$ concentrated in grading $0$. We think of $\cV\subseteq \cV^{\gr}$ as a full
          subcategory via this functor and omit the $(0)$ whenever it is clear from context.
    \item Dually, the terminal map $\Z\to 0$ induces a symmetric monoidal functor
          \[\cV^{\gr} \too \cV, \quad \sI_{\ast} \mapsto \coprod_{n\in \Z} \sI_{n}\]
           which we think of taking a graded object to its \enquote{underlying} object.
           Moreover, this equips $\cV$ with the structure of a $\cV^{\gr}$-algebra, i.e.~a
           locally graded category via the trivial grading.
    \item Barring the symmetric monoidality, the above points apply to any $\cC\in \Prl_{\cV}$
          and we make the same notational conventions and identifications.        
\end{enumerate}
\end{remark}

For any $\cC\in \CAlg(\Prl_{\cV})$, giving a local grading on $\cC$ that is compatible
with the symmetric monoidal structure is equivalent to  a map of commutative $\cV$-algebras 
$\cV[\ZZ]=\cV^{\gr}\to \cC$. By adjunction, this is the same datum as a map of connective
spectra $\ZZ\to \pic(\cC)$. This description of the moduli spaces of local gradings gets
a name.
    \begin{definition}
      For any $\cC \in \CAlg(\Prl_{\cV})$, following~\cite{strictpic} we call the 
      connective spectrum
  \[ \mdef{\pic_{\Z}(\cC)}\coloneq \tau_{\geq 0}\map_{\Sp}(\Z, \pic(\cC)) \in \Spcn\]
  the \tdef{strict Picard spectrum} of $\cC$.
    \end{definition}

Restriction along the map $\SS\to \ZZ$ induces a map $\pic_{\ZZ}(\cC)\to \pic(\cC)$ and
so any strict Picard element $\cL\colon \Z\to \pic(\cC)$ has an underlying
$\otimes$-invertible object, which we abusively\footnote{Beware that, in general
lifting a Picard element to a strict Picard element is additional data.} denote by
the same name. Unwinding the definitions we see that, given a strict Picard element
$\cL \in \pic_{\ZZ}(\cC)$, the induced
grading on $\cC$ is a symmetric monoidal, $\cV$-linear functor, which on objects is given by
the formula
\[\cL^{\otimes}\colon \cV^{\gr} \too \cC,\quad \sI_{\ast} \mapsto \coprod_{n\in \Z} \sI_{n} \otimes \cL^{\otimes n}\,.\]

\begin{notation}\label{not: homotopy groups}
  The category $\cV^{\gr}$ comes equipped with a \enquote{tautological} 
  strict Picard element\footnote{induced by the identity functor $\cV^{\gr}\to \cV^{\gr}$.}  
  which we denote by $\mdef{\1_{\cV}(1)}$.
  Moreover, for any locally graded $\cV$-linear category $\cC$ 
  and an integer $d\in \ZZ$, we write $\mdef{\sI(d)}\coloneq \sI\otimes \1_{\cV}(d)$ 
  and refer to this as the $d$-fold \tdef{Serre-twist} of $\sI\in \cC$. 
  Finally, if $\cV$ is stable, we write
  \[ \mdef{\pi_{t,d} (\sI)} \coloneq \pi_0 \Map(\Sigma^{t}\1_{\cV}(d), \sI)\in \Ab \]
  the graded homotopy groups of any $\sI\in \cC$.
Note that, if the local grading on $\cC$ is induced by a strict Picard element
$\cL\in \pic_{\ZZ}(\cC)$, then we have $\sI(d) \simeq \sI \otimes \cL^{\otimes d}$
for all $\sI\in \cC$.
\end{notation}

We think of a class $\cL\in \pic_{\ZZ}(\cV)$ as a $\otimes$-invertible
element, together with
coherent trivializations of the natural $\Sigma_{n}$-actions on the $\otimes$-powers
$\cL^{\otimes n}$. More precisely, we have the following interaction of strict
Picard elements with free algebras.

\begin{lemma}\label{lem: shearing}
  Let $\cO$ be an operad and $\cC\in \CAlg(\Prl_{\cV})$.
  For any strict Picard element $\cL\in \pic(\cV)_{\ZZ}$, and any $\sI\in \cC$,
  we have a natural equivalence
  \[ \sh^{\cL}\colon 
    \coprod_{n\geq 0}\left( (\sI^{\otimes n} \otimes \cO(n))_{h\Sigma_{n}} \otimes \cL^{\otimes n}\right)
    \xrightarrow{\sim} U_{\cC}(\1_{\cC}\{\sI\otimes \cL\}_{\cO})
  \,.\]
\end{lemma}
\begin{proof}
Consider the free graded $\cO$-algebra $\1_{\cC}\{\sI(1)\}_{\cO}$
 on the Serre-twist of $\sI$. 
By the formula for the free $\cO$-algebra~\cite[\href{https://www.math.ias.edu/~lurie/papers/HA.pdf\#theorem.3.1.3.13}{3.1.3.13}]{ha}
we have an equivalence 
\[ U_{\cC}( \1_{\cV}\{\sI(1)\}) \simeq \coprod_{n \geq 0} (\sI^{\otimes n}(n) \otimes \cO(n))_{h\Sigma_n}\,.\]
Thus, the claim is immediate by applying the symmetric monoidal functor $\cC^{\gr}\to \cC$
induced by the strict Picard element $\cL$. 
\end{proof}

This equivalence gives rise to the \emph{shearing construction} on classes in the free
algebra, which we want to spell out more explicitly.

\begin{construction}\label{cons: shearing}
Given a strict Picard element $\cL\in \pic_{\ZZ}(\cV)$ and any class in the free algebra
\[ Q\colon \sJ(n) \too \1_{\cV}\{\sI(1)\}_{\cO}\in \cV^{\gr}\]
in grading $n$, we may use the equivalence of \cref{lem: shearing} to obtain a class
\[ \mdef{\sh^{\cL}(Q)} \colon \sJ(n) \otimes \cL^{\otimes n} \too \1_{\cV}\{\sI(1)\otimes \cL\}\in \cV^{\gr}\]
which we refer to as the \tdef{shearing} of $Q$ by $\cL$. Moreover, by using the inverse
$\cL^{-1}$, we see that this defines a bijection
between $\sJ$-based classes in grading $n$ in the free algebra on $\sI$ and
$\sJ\otimes \cL^{\otimes n}$-based classes in grading $n$ in the free algebra on $\sI\otimes \cL$.
\end{construction}

The situation we are most interested in comes with an abundance of strict units,
which is the following lemma.

\begin{lemma}\label{lem: suspension hstrict}
  For any field $k$, the 2-fold suspension $\Sigma^{2}k\in \Mod_{k}$ admits a refinement
  to strict Picard element, which is unique up to a choice of unit $u \in k^{\times}$
  Moreover, if $\mathrm{char}(k)=2$, the 1-fold suspension $\Sigma k \in \Mod_{k}$
  refines to a strict Picard element, unique up to a unit of $k$.
\end{lemma}
\begin{proof}
  Since $k$ is a field, the Picard spectrum has homotopy groups concentrated
  in two degrees, namely
  \[ \pi_{0}\pic(\Mod_{k}) = \Z\quad \text{and}\quad\pi_{1}\pic(\Mod_{k}) = k^{\times} \,.\]
  By~\cite[\href{https://arxiv.org/pdf/2208.03073\#thm.3.8}{Proposition 3.8}]{strictpic} 
  combined 
  with~\cite[\href{https://arxiv.org/pdf/2208.03073\#thm.3.2}{Proposition 3.2}]{strictpic},
  we have $\pi_1\pic_{\ZZ}(\Mod_k)=k^{\times}$. Moreover, $\pi_0$
  is given by the kernel of the symmetric monoidal dimension map
  \[ \mathrm{dim}_{k}\colon \pi_{0}\pic(\Mod_{k})\to k^{\times}\,,\]
  so the claim follows.
\end{proof}

  Consider the map of commutative monoids 
  $\tau\colon \mathbb{N} \to \Omega^{\infty}\pic_{\ZZ}(\cV^{\gr})$
  picking out the tautological strict unit $\1_{\cV}(1)$.
  Taking the Thom-construction, we obtain a commutative algebra in $\cV^{\gr}$, 
  which we denote by
  \[\mdef{\1_{\cV}[\tau]}\coloneq \colim_{n\in \ZZ}\1_{\cV}(1)^{\otimes n} \in \CAlg(\cV^{\gr}).\]

\begin{definition}
  For any $\cC\in \Prl_{\cV}$ the category of \tdef{filtered objects}
  in $\cC$ is defined to be the relative tensor product
  \[ \mdef{\cC^{\fil}} \coloneq \Mod_{\1_{\cV}[\tau]}(\cV^{\gr})\otimes_{\cV^{\gr}} \cC^{\gr}
  \in \Prl_{\cV^{\gr}}\,.\]
\end{definition}

By construction, $\cV^{\fil}$ is a presentably symmetric monoidal, $\cV$-linear
category equipped with a natural local grading given by the base change functor
\[ -\otimes \1_{\cV}[\tau]\colon \cV^{\gr} \too \cV^{\fil}\,,\]
which we call the \tdef{$\tau$-grading} on $\cV^{\fil}$.
Unwinding the definitions, we extract from the $\1_{\cV}[\tau]$-action
on any $\sI\in \cC^{\fil}$ a natural map
\[ \tau \colon \sI(1) \too \sI \,,\]
i.e.~a degree one map of the underlying graded objects.
Accordingly, for any $\cC\in \Prl_{\cV}$, we depict the objects of $\cC^{\fil}$ as
integer indexed sequences
 \[ \cdots \to \sI_{-2} \to \sI_{-1} \to \sI_0 \to \sI_1 \to \sI_2 \to \cdots \]
 where $\sI_i \in \cC$.

\begin{remark}
 By~\cite[Proposition 3.16]{rotation}, inverting the class 
 $\tau\colon \1_{\cV^{\fil}}(1)\to \1_{\cV^{\fil}}$ gives us a 
 commutative algebra $\1_{\cV^{\fil}}[\tau^{-1}] \in \CAlg(\cV^{\fil})$ such that the composite
functor
\[ \cV\too \cV^{\fil} \xrightarrow{\tau^{-1}}  \Mod_{\1_{\cV^{\fil}}[\tau^{-1}]}(\cV^{\fil})\,.\]
is an equivalence of categories. Under this identification, for any $\sI \in \cV^{\fil}$
the $\tau$-localization is given by taking the colimit along the filtration
\[ \tau^{-1}\sI = \fcolim\left(\cdots \to \sI_{-1}\to \sI_{0}\to \sI_{1}\to \cdots\right)
  \in \cV\,.\]
\end{remark}


\section{Power operations and decomposing free \texorpdfstring{$\E_{n}$}{En}-algebras}\label{sec: powerop}
Let $\cC$ be a presentably monoidal category and let $\cO$ be an operad.
Given $\sI,\sJ\in \cC$, some class in the free $\cO$-algebra
\[ Q\colon \sJ \to \1_{\cC}\{\sI\}_{\cO}\]
and an $A\in \Alg_{\cO}(\cC)$, we get a natural map of sets
\[ Q(-)\colon \pi_0\Map_{\cC}(\sI,A)\to \pi_0\Map_{\cC}(\sJ,A)\]
defined by taking a class $v\colon \sI \to A $ to the composite map
\[ Q(v)\colon \sJ \xrightarrow{Q} \1_{\cC}\{\sI\}_{\cO}
  \xrightarrow{\1_{\cC}\{v\}_{\cO}} A\,.\]
On typically considers the case where $\sI,\sJ$ are shifts of a compact generator. Moreover,
as we used in \cref{lem: shearing}, there is a natural filtration on the free
algebra given by the \emph{arity} of the classes. Let us be precise about what we mean.

\begin{definition}
  Let $\cC \in \CAlg(\Prl)$ and let $\cO$ be an operad.
An $\cO$-algebra \tdef{power operation} in $\cC$ of
\tdef{type $(t,s)$} and \tdef{arity $d$} is a map of graded objects
\[ \mdef{Q}\colon \Sigma^{s}\1_{\cC}(d) \too \free{\1}{\cO}{\Sigma^{t}\1_{\cC}(1)}
  \in \cC^{\gr}\,.\]
We also refer to \mdef{$s-t$} as the \tdef{total degree} of $Q$.
\end{definition}

For a prime $p$ and some $1\leq n \leq \infty$, we discuss $\E_n$-power operations in 
$\F_p$-linear categories for $p$ a prime.

\subsection{Operations and \texorpdfstring{$\E_k$}{Ek}-decompositions for \texorpdfstring{$p=2$}{p=2}}

We begin by describing the $\E_{n}$-power operations in $\Mod_{\F_{2}}$ for any $0\leq n\leq \infty$.
The construction is originally Kudo and Araki~\cite{ArakiKudo} and Browder~\cite{browder}.
We briefly review it in our language, while also keeping track of the arity filtration by working
with graded algebras.

Recall from \cref{lem: suspension hstrict} that for each $t\in \ZZ$ the suspension
$\Sigma^{t}\F_2$ canonically refines to a strict Picard element. We denote the associated
shearing map of \cref{cons: shearing} simply by \mdef{$\sh^{(t)}$}.

\begin{construction}\label{cons: even operations}
  Let $1\leq n \leq \infty$, then the space $\E_n(2)$ of configurations of two points in
  $\mathbb{R}^n$ is $C_2$-equivariantly homotopic to $S^{n-1}$ with the flip action.
  In particular, the $\E_n$-operations of arity 2 on a class in topological degree $0$ and
    grading $d$ are parametrized by the homology of real projective space
\[ (\F_{2}(d)^{\otimes 2} \otimes \E_n(2))_{hC_{2}}\simeq \F_{2}(2d)\otimes \mathbb{R}P^{n-1} \simeq \bigoplus_{0\leq i\leq n-1} \Sigma^{i}\F_{2}(2d)\in \Mod_{\F_{2}}^{\gr}\,.\]
Thus, for each $0\leq i \leq n-1$ we obtain a canonical arity 2 operation of type
$(0,i)$ which we denote by
\[\mdef{Q_{i}}\colon \Sigma^{i}\F_{2}(2d)\too \free{\F_{2}}{n}{\F_{2}(d)}.\]
Thus, for any $t\in \Z$, we obtain a canonical arity 2 operation $\sh^{(t)}Q_{i}$ of type $(t,i+2t)$.
We often suppress the shearing
in our notation, meaning for an element $v \in \pi_{\ast}A$ of some $A\in \EAlg{n}{\Mod_{\F_2}}$,
we define
\[ \mdef{Q_{i}(v)}:= \sh^{(\abs{v})}Q_{i}(v).\]
If $\cC$ is any presentably $\E_n$-monoidal, locally graded, $\F_2$-linear category,
we obtain corresponding classes $\sh^{(t)}Q_i$ in the free algebra
$\1_{\cC}\{\Sigma^{t}\1_{\cC}(d)\}_n$ via the $\E_n$-monoidal unit functor
$\Mod_{\F_2}^{\gr}\to \cC$, for which we use the same notation.
\end{construction}

\begin{remark}
 This is the \textit{lower indexing} convention for the even primary power operations.
  The \textit{upper indexing} is obtained by setting 
  $Q^{i}(v)=Q_{i-\abs{v}}(v)$,
  with the convention that this vanishes whenever the right hand side is not defined.
  This has the advantage that $Q^{i}$ always has total degree $i$, but for us the lower
  indexing will be more convenient.
\end{remark}

\begin{proposition}\label{prop: even operations}
  Let $\cC$ be a presentably $\E_n$-monoidal, locally graded, $\F_2$-linear 
  category and let $A\in \EAlg{n}{\cC}$. For any $i<n-1$, the operations
  \[ Q_i\colon \pi_{\ast,\ast} A \too \pi_{2\ast +i,2\ast}A,\]
  of \cref{cons: even operations} satisfy the following identities:
  \begin{enumerate}
    \item Frobenius: $Q_0(v)=v^2$ for all $v\in \pi_{\ast,\ast} A $
    \item Unitality: $Q_i(1)=0$ for $i>0$
    \item Additivity: $Q_i$ is an $\F_2$-linear map
    \item Cartan formula: For all $v,w\in \pi_{\ast,\ast}A$ we have that
          \[ Q_i(vw)= \sum_{k+l=i} Q_k(v)Q_l(w)\]
  \end{enumerate}
\end{proposition}
\begin{proof}
  The Frobenius and Unitality relations are straightforward from the definitions.
  The additivity and Cartan formula can be checked on the free algebra by analyzing
 the map
  \[\1_{\cC}\{u\}_n\too\1_{\cC}\{v,w\}_n,\quad  u \mapsto v+w\]
  classifying addition, as well as the map
  \[ \1_{\cC}\{u\}_n\too \1_{\cC}\{v,w\}_n, \quad u \mapsto vw\]
  classifying multiplication, respectively. Since these relations are preserved by
  the unit functor $\Mod_{\F_2}^{\gr}\to \cC$, we can further reduce to the
  case $\cC=\Mod_{\F_2}^{\gr}$, which is entirely classical, see for
  example~\cite[III.~Theorem 3.1]{hring}.
\end{proof}

We want to give a coherent decomposition of the free graded $\E_n$-algebra in any $\F_2$-linear
category. To do this, we fix the following setup for the remainder of this section.

\begin{setting}\label{set: even prime}
  Fix $0\leq n \leq \infty$ and let $\cC$ be a presentably $\E_{n+1}$-monoidal,
  locally graded, $\F_2$-linear category. Let $t$ and $d$ be integers and let $x$
  be a generator in topological degree $t$ and grading $d$.
\end{setting}

The computation is essentially due to Araki--Kudo in~\cite{ArakiKudo}
where they give a description of the homology $\F_2[\Omega^n \Sigma^n S^t]$ 
as a polynomial algebra on the iterated operations $Q_i$. We observe that the proof given by
Browder~\cite{browder} easily upgrades to a highly coherent statement.

For integers $0\leq a\leq b <\infty$ and a sequence of non-negative integers
$I=(i_{a},i_{a+1}\dots, i_{b})$, we write
\[\mdef{Q_I}\coloneq Q_{a}^{i_a}Q_{a+1}^{i_{a+1}}\cdots Q_{b}^{i_b}\]
for the composite operation, where by definition $Q_k^{0}=\id$ for any $k$.

\begin{proposition}\label{prop: En decomposition F2}
  In \cref{set: even prime} we have for any $0\leq k< n+1$ a natural equivalence
  \[ \bigotimes_{I} \free{\1_{\cC}}{k}{Q_I(x)}
    \rar{\sim}\free{\1_{\cC}}{n+1}{x} \in \EAlg{k}{\cC^{\gr}}\,,\]
    of $\E_{k}$-algebras in $\cC^{\gr}$, where the infinite tensor product 
    is indexed over all sequences of non-negative integers $I=(i_{k+1},\dots, i_n)$.
\end{proposition}
\begin{proof}
  Since $\cC$ comes equipped with a colimit preserving $\E_{n+1}$-monoidal
  functor $\Mod_{\F_{2}}\to \cC$, it suffices to show the claim
  in the universal case $\cC=\Mod_{\F_2}$. The $n=\infty$ case follows readily by taking
  the filtered colimit over the finite cases, so we may assume $n<\infty$.

  We argue using a simultaneous induction on $n$ and $k$. First, note that the case
  $k=0$ and $n=1$ is immediate from the fact that for any $\sI\in \Mod^{\gr}_{\F_2}$ we have the formula
  \[ \F_2\{\sI\}_{1} \simeq \bigoplus_{m \geq 0} \sI^{\otimes m}\]
  for the underlying module of a free $\E_1$-algebra, i.e.~the homotopy groups are polynomial
  on any $\F_2$-linear basis of $\sI$.

  Thus, we assume that $k=n$ and that we have shown the claim for all pairs $(k^{\prime},n^{\prime})$
  with $k^{\prime} \leq n^{\prime}<n$.
  In this case, the only operations in the allowed range are the iterates $Q_{n}^{i}$ for $i\geq 0$.
  By construction of the $Q^i_{n}$, we have natural maps of graded $\E_{n}$-algebras
  \[ q^i\colon \free{\F_2}{n}{Q_{n}^i(x)}\to \free{\F_2}{n+1}{x}.\]
  By Dunn-additivity~\cite[\href{https://www.math.ias.edu/~lurie/papers/HA.pdf\#theorem.5.1.2.2}{5.1.2.2}]{ha},
  we know that for any $j\geq 0$ the multiplication map
  \[\mu_{j}\colon \free{\F_2}{n}{x}^{\otimes j}\too \free{\F_2}{n}{x}\]
  is a map of $\E_n$-algebras and so the composite
  \[ A_{j}\coloneq\bigotimes_{0\leq i \leq j} \free{\F_2}{n}{Q_n^i(x)}
  \xrightarrow{\otimes q^{i}}
    \bigotimes_{0\leq i \leq j} \free{\F_2}{n}{x}
    \xrightarrow{\mu_{j}} \free{\F_2}{n+1}{x},\]
  defines a map of $\E_{n}$-algebras in $\Mod_{\F_{2}}^{\gr}$ as well.
  Taking the filtered colimit over $j$, we obtain the desired comparison map 
  $q\colon A= \colim_j A_j\to \F_2\{x\}_{n+1}$. Since the construction of the
  map $q$ is compatible with the shearing functor of \cref{lem: shearing},
 we may assume that $x$ is in positive topological degree by \cref{lem: suspension hstrict}.
 Moreover, the functor that forgets the grading and taking homotopy groups are both conservative,
 so it suffices to show that $q$ induces an equivalence on (un-graded) homotopy groups.
 Using the identification of the free $\E_{m}$-algebra on class in positive degree $t$
 with the homology of $\Omega^m\Sigma^m S^t$ of \cref{lem: free algebra loop space},
 we obtain a comparison map of homologies
 \[ \pi_{\ast}q \colon \bigotimes_{i\geq 0} H_{\ast}(\Omega^{n}\Sigma^n S^{\abs{Q_n^i(x)}};\F_2)
 \too H_{\ast}(\Omega^{n+1}\Sigma^{n+1}S^{t};\F_2)\,.\]
By the inductive assumption, the left hand side is given by an infinite polynomial algebra
on classes of the form $Q_{1}^{i_{1}}\dots Q_{n-1}^{i_{n-1}}(x)$ and hence the map
is an equivalence, for example by the computation of the right hand side
in~\cite[Theorem 3]{browder}.
\end{proof}

As a consequence, we get the following decomposition of the unit in any $\E_{n+1}$-monoidal
$\F_2$-linear category.

\begin{corollary}\label{cor: unit decomp F2}
  In \cref{set: even prime} with $1 \leq k\leq n$, the augmentation map
  $\free{\1_{\cC}}{n+1}{x}\to \1_{\cC}$ taking $x$ to $0$ induces an equivalence
of $\E_{k-1}$-algebras
  \[  \bigotimes_{I} (\1_{\cC}\{x\}_{n+1} \otimes_{\1_{\cC}\{Q_I(x)\}_{k}} \1_{\cC})
    \xrightarrow{\sim} \1_{\cC} \in \EAlg{k-1}{\Mod_{\1_{\cC}\{x\}_{n+1}}(\cC^{\gr})}\,, \]
  where $I$ runs over all sequences of non-negative integers $I=(i_{k+1},\dots, i_n)$.
\end{corollary}
\begin{proof}
  Observe that, by the relation $Q_I(0)=0$ of \cref{prop: even operations} the comparison map
  of \cref{prop: En decomposition F2} upgrades to an equivalence of augmented algebras if we
  equip all free algebras in sight with the $0$-augmentation. Thus, we may iteratively apply
  \cref{cons: general decomp} to the finite stages of the decomposition in
  \cref{prop: En decomposition F2} and take the filtered colimit to get the desired equivalence.
\end{proof}

With this in hand, we may apply \cref{lem: En relative bar} to obtain a filtered decomposition
of $\E_n$-cofibers.

\begin{corollary}\label{cor: filtered cofiber p=2}
  In \cref{set: even prime} with $1\leq k \leq n$, for any $A\in \EAlg{n+1}{\cC}$ and any
  class $v\colon \Sigma^{t}\1_{\cC} \to A$ we have an equivalence of filtered $\E_{k-1}$-algebras
  \[ \bigotimes_{I} \modEn{A}{k-1}{Q_{I}(\tau v)}
    \xrightarrow{\sim}\modEn{A}{n}{\tau v}\in \EAlg{k-1}{\cC^{\fil}}\]
  where $I$ runs over all sequences of non-negative integers $I=(i_{k+1},\dots, i_n)$.
\end{corollary}
\begin{proof}
  Applying \cref{cor: unit decomp F2} to the category $\cC^{\fil}$ with the $\tau$-grading,
  we get an equivalence
\[ \1_{\cC}\simeq \bigotimes_I (\1_{\cC} \otimes_{\1_{\cC}\{Q_I(\tau v)\}_{k}} \1_{\cC}\{\tau v\}_{n+1}) \]
in the category $\Mod_{\1_{\cC}\{\tau v\}_n}(\EAlg{k-1}{\cC^{\fil}})$. Base changing along
the map $\1_{\cC}\{\tau v\}_{n+1} \to A$ induced by $v$, we learn that we have an equivalence of filtered
$\E_{k-1}$-algebras
\[ \1_{\cC} \otimes_{\1_{\cC}\{\tau v\}_{n+1}} A
\simeq  \bigotimes_{I}( \1_{\cC} \otimes_{\1_{\cC}\{Q_I(\tau v)\}_{k}}A )\,.\]
Thus, we may apply \cref{lem: En relative bar} to identify the relative tensor products with
$\E_{n}$ and $\E_{k-1}$ cofibers respectively, which yields our claim.
\end{proof}

\subsection{Operations and \texorpdfstring{$\E_k$}{Ek}-decompositions for \texorpdfstring{$p>2$}{p>2}}\label{ssec: Odd prime op}

For odd primes $p>2$, the situation is more subtle, but still uniform.
The following operations were originally constructed by Dyer and Lashof in~\cite{dyerlash}. We briefly
review them, roughly following the later definition of Cohen in~\cite{cohen}, while keeping track of the grading by arity.

Throughout this section, we consider the refinement of 
$\Sigma^t \F_p$ to a strict Picard element for any even $t\in \ZZ$, as induced by the 
unit $1\in \F_p^{\times}$ via \cref{lem: suspension hstrict}. We denote the associated
shearing map of \cref{cons: shearing} simply by $\mdef{\sh^{(t)}}$.

\begin{construction}\label{cons: odd operations}
  Let $2\leq n \leq \infty$ and $p$ be an odd prime. Then, for a generator in
  topological degree 0 and grading $d$, by the computation of $\E_n(p)\otimes \F_p$
  in~\cite[Theorem 5.2]{cohen} there are classes in grading $pd$
  \[ \mdef{\beta^{\eps}P_i} \colon \Sigma^{2i(p-1)-\eps}\F_p(pd)\too \F_p\{\F_p(d)\}_n\]
  for any integer $1\leq i\leq \frac{n-1}{2}$ and $\eps\in \{0,1\}$ as well as the degree zero operation given
  by the $p$-th power $\mdef{P_0}=(-)^p$. Thus, we obtain sheared operations $\sh^{(t)}(\beta^{\eps}P_{i})$ of type $(t, pt + 2i(p-1)-\eps)$
  for any even integer $t$.

  Similarly, for a generator in topological degree 1 we have classes
  \[\mdef{\beta^{\eps}P_{j-\frac{1}{2}}} \colon \Sigma^{p+(2j-1)(p-1)-\eps}\F_{p}(pd) \too
    \free{\F_{p}}{n}{\Sigma \F_{p}(d)}\]
  for any integer $1\leq j\leq \frac{n-1}{2}$
  and thus obtain sheared arity $p$ operations $\sh^{(t)}(\beta^{\eps}P_{j-\frac{1}{2}})$
  of type $(t+1, p(t+1)+ (2j-1)(p-1))$ for any even integer $t$.

 We again suppress the shearing from the notation and define operations acting on \emph{even}
 classes
 \[ \mdef{\beta^{\eps}P_{i}(v)}\coloneq
   \sh^{(\abs{v})}\beta^{\eps}P_{i}(v),\quad \abs{v}~\text{even}
   \]
   as well as operations acting on \emph{odd} classes
   \[ \mdef{\beta^{\eps}P_{j-\frac{1}{2}}(w)}\coloneq \sh^{(\abs{w}-1)}\beta^{\eps}P_{j-\frac{1}{2}}(w), \quad \abs{w}~\text{odd}\,,\]
   with the convention that the operations act by $0$ whenever the parities do not match.

   As in the $p=2$ case of \cref{cons: even operations}, these operations are defined in
   any presentably $\E_n$-monoidal, locally graded, $\F_p$-linear category.
\end{construction}

\begin{remark}\label{rmk: half integer notation}
  The purpose of the half-integer notation is two-fold.
  First, as in the $p=2$ case, there also exists an upper indexing convention, defined such
  that $\beta^{\eps}P^{i}$ has total degree $2i(p-1)-\eps$ which is given by setting
  \[ \beta^{\eps}P^i(v)\coloneq \beta^{\eps}P_{i-\frac{\abs{v}}{2}}(v)\]
  Second, writing $\frac{1}{2}\ZZ\subseteq \mathbb{Q}$ for the additive subgroup of half-integers,
  for any $i\in \frac{1}{2}\ZZ$ the operation $\beta^{\eps}P_i$ acting on a class
  in (not-necessarily even) topological degree $t$ has type
  $(t, pt + 2i(p-1)-\eps)$, so we do not need to explicitly differentiate between the action
  on odd or even degree classes.
\end{remark}

\begin{proposition}\label{Odd operations}
  Let $\cC$ be a presentably $\E_n$-monoidal, locally graded, $\F_p$-linear category.
  For any $A\in \CAlg(\cC)$ and any pair $(i,\eps)$ where $i$ is half integer such that
  $0\leq i <\frac{n-1}{2}$ and $\eps \in \{0,1\}$, the operations
  \[\beta^{\eps}P_i\colon \pi_{\ast,\ast}A\too \pi_{p\ast + 2i(p-1)-\eps,p\ast}A\]
  defined in \cref{cons: odd operations} satisfy the following relations:
  \begin{enumerate}
    \item Frobenius: $P_0(v)=v^p$ for all $v\in \pi_{\ast,\ast} A$
    \item Unitality: $\beta^\eps P_i(1)=0$ for all $i> 0$.
    \item Additivity: $\beta^\eps P_{i}(v+w) =\beta^{\eps} P_{i}(v) + \beta^{\eps}P_{i}(w)$
          for all $v,w\in \pi_{\ast,\ast}A$.
    \item Cartan Formula:
          \[P_i(vw) = \sum_{k+l=i} P_k(v) P_l(w)\]
          \[ \beta P_i(vw)= \sum_{k+l=i}\left( \beta P_k(v) P_l(w) + (-1)^{\abs{x}} P_k(v)\beta P_l(w)\right) \]
          for all $v,w\in \pi_{\ast,\ast}A$.
  \end{enumerate}
\end{proposition}
\begin{proof}
    The Frobenius and Unitality relations are straightforward from the definitions.
    For the remaining two, we may argue as in \cref{prop: even operations} to reduce
    to the universal case of $\cC=\Mod_{\F_p}^{\gr}$. In this case, the claims are shown
    in Theorem 2.2 and Theorem 2.3 in~\cite{dyerlash} respectively.
\end{proof}

We also adopt and expand the following terminology of~\cite{dyerlash} to our indexing convention.

\begin{definition}\label{defn: allowable DL}
  A \tdef{Dyer--Lashof} sequence $I=(\eps_n,i_n, \dots, \eps_1,i_1)$ is a
  sequence of half-integers with $i_m \in \frac{1}{2}\ZZ_{\geq 1}$ and $\eps_m \in \{0,1\}$ for all $m$.
  For any $1\leq k \leq \infty$, the composite of Dyer--Lashof operations
  $P_I(x)=\beta^{\eps_n}P_{i_{n}} \cdots \beta^{\eps_{1}}P_{i_{1}}(x)$ applied to a class $x$
  in topological degree $t$ is called \tdef{$k$-allowable} if the following hold:
\begin{enumerate}
  \item $i_1$ is an integer if and only if $t$ is even.
  \item We have $i_{m+1}\leq i_{m}$ for all $m$.
  \item If $\eps_m=1$ for any $k$, then $i_m-i_{m+1}$ is a half integer.
  \item We have $i_{m}\leq \frac{k-1}{2}$ for all $m$.
\end{enumerate}
For any $n\leq k$, we say that a $k$-allowable sequence $P_I(x)$ is \tdef{$n$-bounded}
if $i_m> \frac{n-1}{2}$ for all $m$ and write $I \gtrsim n$ in this case.
\end{definition}

The definition above precisely excludes composites that are automatically zero
or to which an Adem relation could be applied.
For example, the sequences $P_1P_1, P_{\frac{1}{2}}\beta P_1$ and $\beta P_{\frac{1}{2}}\beta P_1$
are allowable. With this terminology in hand, we get a description of the homotopy
groups of the free $\E_k$-algebra from the computations of~\cite{dyerlash}.

\begin{lemma}\label{lem: Free En odd prime homotopy}
  Let $x$ be a generator in topological degree $t\in \ZZ$. For any $k\geq 0$ such
  that $t$ and $k$ have different parity, $\pi_\ast \free{\F_p}{k}{x}$ is freely generated
  as a discrete, graded commutative $\F_p$-algebra by the classes $P_I(x)$ given
  by all $k$-allowable composite operations.
\end{lemma}
\begin{proof}
  Since the two-fold suspension $\Sigma^2\F_p$ refines to a strict Picard element
  by \cref{lem: suspension hstrict}, the shearing equivalence of \cref{lem: shearing}
  combined with the construction of the operations in \cref{cons: odd operations} tells us
  that it suffices to prove the claim for $t>0$. By \cref{lem: free algebra loop space}
  the statement reduces to the computation of the homology
  $\pi_{\ast}\F_p[\Omega^k S^{k+t}]=H_{\ast}(\Omega^kS^{k+t};\F_p)$ done in~\cite[Theorem 5.2]{dyerlash}.
\end{proof}

We are now ready to state an odd primary analogue of \cref{prop: En decomposition F2}.
For simplicity and since it as all we will need, we restrict ourselves to the
decomposition of a free $\E_{\infty}$-algebra.
  
\begin{proposition}\label{prop: En decomposition Fp}
Let $\cC$ be a presentably symmetric monoidal, locally graded, $\F_p$-linear category
and $x$ a generator in topological degree $t$ and grading $d$.
For any $k\geq 0$ such that $t$ and $k$ have different parity,
we have a natural equivalence of graded $\E_k$-algebras
\[ \left(\bigotimes_{I\gtrsim k} \1_{\cC}\{P_I(x)\}_k \right) \otimes
\left(\bigotimes_{J\gtrsim k+1} \1_{\cC}\{P_J(x)\}_{k+1}\right)
\xrightarrow{\sim} \1_{\cC}\{x\}_{\infty}\,,\]
where $P_I(x)$ runs over all $\infty$-allowable, $k$-bounded composites which preserve parity
and $P_J(x)$ runs over all $\infty$-allowable, $k+1$-bounded composites which flip parity.
\end{proposition}
\begin{proof}
We may construct a natural comparison map of graded $\E_k$-algebras as in the
proof of \cref{prop: En decomposition F2} and only consider the universal
case $\cC=\Mod^{\gr}_{\F_p}$.
Since the functor which forgets the grading is conservative, we further
reduce to the ungraded statement.
Moreover, the shearing equivalence of \cref{lem: shearing} combined
with the strictness of even suspensions \cref{lem: suspension hstrict}
tells us that we can assume $t>0$. Using the presentation of
the free $\E_n$-algebra via loop spaces  discussed in 
\cref{lem: free algebra loop space} and the fact we can check equivalences
on homotopy groups, the map being an equivalence becomes precisely
the statement of~\cite[Theorem 5.1]{dyerlash}
computing the homology $H_{\ast}(\Omega^{\infty}\Sigma^{\infty}S^{t};\F_p)$ combined with
\cref{lem: Free En odd prime homotopy}.
\end{proof}

\begin{corollary}\label{cor: unit decomp Fp}
  Let $\cC$ be a presentably symmetric monoidal,
  $\F_{p}$-linear category and let $A\in \CAlg(\cC)$. For any $t\in \ZZ$ such
  that $t$ and $k$ have different parity and a class
  $v\colon \Sigma^{t}\1_{\cC} \to A$, we have a natural equivalence
  \[ \bigotimes_{I \gtrsim k} \modEn{A}{k}{P_I(\tau v)} \otimes
  \bigotimes_{J\gtrsim k+1} \modEn{A}{k+1}{P_J(\tau v)} \xrightarrow{\sim}
  \modEn{A}{\infty}{\tau v} \]
of filtered $\E_{k}$-algebras in $\cC$.
\end{corollary}
\begin{proof}
  Argue exactly as in the proof of \cref{cor: filtered cofiber p=2} using
  \cref{prop: En decomposition Fp}.
\end{proof}


\section{Points in positive characteristic}\label{sec: points}
Having done our homework on Dyer--Lashof operations, we can now turn
towards discussing whether $\E_\infty$-rings admit reduced points.

We begin by reviewing the machinery of nullstellensatzian rings and nilpotence detecting
maps of~\cite{cnst}, which forms the basis of our arguments.

At the prime $p=2$, we show that the universal map killing an element which squares 
to 0 detects nilpotence in \cref{prop: killing nilpotence}. From this, we deduce that all 
nullstellensatzian $\F_2$-algebras must be separably closed, graded fields and so our 
main theorem in the form of \cref{thm: existence of reduced points} follows.

For $p>2$, we show in \cref{prop: bosonic nilpotent} that the free commutative $\F_p$-algebra 
on a $p$-nilpotent generator $\modEn{\F_p\{x\}}{\infty}{x^p}$ contains non-nilpotent classes, 
given by certain parity preserving composite operations $P_I(x)$. 
As am immediate consequence in \cref{thm: odd prime}, we obtain that
the localization of this ring admits no maps to any
$\E_\infty$-ring with reduced or even homotopy groups.
By further analyzing the interaction of nilpotent classes with the Dyer--Lashof operations,
we also deduce some properties of nullstellensatzian $\F_p$-algebras in \cref{prop: odd prime nstz description}.

\subsection{Geometric points and detecting nilpotence}\label{ssec: nstzian rings}

Let us begin by recalling the definition of nullstellensatzian objects of~\cite{cnst}
in an arbitrary category.

\begin{definition}
  Let $\cC$ be a category and $\alpha$ be a regular cardinal. An object $X \in \cC$ is
  called \tdef{$\alpha$-nullstellensatzian} if every $\alpha$-compact object
  $(f\colon X\to Y)\in (\cC_{X/})^{\alpha}$ admits a section $Y \to X$. For $\alpha = \omega_{0}$ we call $X$
  just \tdef{nullstellensatzian}.
\end{definition}

By the weak form of Hilbert's Nullstellensatz, the nullstellensatzian objects
in the 1-category of ordinary commutative rings $\CAlg(\rm{Ab})$ are precisely
the algebraically closed fields, which is the reason for the terminology.

A higher algebra variant of the Nullstellensatz was proven by Burklund--Schlank--Yuan
for the monochromatic category.
More precisely~\cite[\href{https://arxiv.org/pdf/2207.09929\#theorem.1.1}{Theorem A}]{cnst}
shows that, in the category $T(n)$-local $\E_\infty$-rings for some $n\geq 0$,
the nullstellensatzian objects are precisely those given by the
Lubin--Tate theory $E_n(F)$ attached to an algebraically closed field $F$.

Moreover, they show that nullstellensatzian objects exist in any well behaved
category of commutative algebras. We are in an even nicer situation, namely
our categories are rigid in the following sense.

\begin{definition}
  A compactly generated category $\cC\in \CAlg(\Prl)$ is called \tdef{rigid}
  if every compact object is dualizable and the monoidal unit
  $\1_{\cC}\in \cC$ is a compact object.
\end{definition}

Note that, $\1_{\cC}\in \cC$ being compact implies that \textit{every} dualizable
object of $\cC$ is compact. One can define rigid categories in greater generality, but
for us rigid will always mean compactly generated.

\begin{example}
  The category of spectra $\Sp$ is rigid. Moreover, for any
  rigid category $\cC$ and any $A\in \CAlg(\cC)$, the category $\Mod_A(\cC)$ is again
  rigid.
\end{example}

The nullstellensatzian objects in a category of algebras get a less cumbersome name.

\begin{terminology}
Let $\cC$ be a symmetric monoidal category and $A\in \CAlg(\cC)$. A \tdef{geometric point}
of $A$ is a nullstellensatzian object $T\in \CAlg_A(\cC)$.
\end{terminology}

Note that we slightly deviate from the terminology of~\cite{cnst}, where a geometric
point is an equivalence class of nullstellensatzians in $\CAlg_A(\cC)$ under the relation
that $T_1\sim T_2$ if $T_1 \otimes T_2$ is non-terminal. We then have the following
very general existence result.

\begin{proposition}[{\cite[\href{https://arxiv.org/pdf/2207.09929\#nul.A.17}{Proposition A.17}]{cnst}}] \label{existence of nstzians}
  Let $\cC$ be a stable rigid category.
  For any regular cardinal $\alpha$ and any non-zero $A\in \CAlg(\cC)$, there exists
  an $\alpha$-nullstellensatzian object $T\in \CAlg_{A}(\cC)$. In particular,
  $A$ admits a geometric point.
\end{proposition}

Since we assume compact generation, we immediately learn the following about detecting
non-nilpotent classes.

\begin{corollary}\label{cor: detecting classes nstzian}
  Let $\cC$ be a stable, rigid category and $A\in \CAlg(\cC)$.
  For any non-nilpotent class $v\colon \Sigma^t\1_{\cC}\to A$
  there exists a geometric point $f\colon A \to T$ such that $f(v)$ is invertible.
\end{corollary}
\begin{proof}
  Indeed, since $v$ is not nilpotent and $\1_{\cC}\in \cC^{\omega}$ by assumption,
  the localization $\1_{\cC}\to A[v^{-1}]$ is non-zero,
  and hence admits a geometric point $T$, so the claim follows.
\end{proof}

In particular, we learn that, all non-nilpotent classes of a nullstellensatzian commutative
ring are already invertible.\footnote{Beware that, if one does not assume $\1_{\cC}\in \cC^{\omega}$, the notion of a class
  being nilpotent is dependent on the ambient category. For example, element $p\in \pi_{0}\SS^{\wedge}_{p}$ is not nilpotent in the $\E_{\infty}$-ring
  $\SS^{\wedge}_{p}$, but the localization at $p$ is zero in the category of $p$-complete spectra.}
Hence, what remains is the much more subtle task of analyzing the behavior of nilpotent elements.

To do this, we recall some definitions and useful facts
discussed in~\cite[\href{https://arxiv.org/pdf/2207.09929\#section.4}{Section 4}]{cnst}, and prove
a closure property of nilpotence detecting $\E_0$-algebras.
For any $A\in\cC$, we say that a map $v\colon \sI\to \sJ\in \cC$ is \tdef{$\otimes$-nilpotent at $A$}
if $v^{\otimes n} \otimes A\simeq 0$ for some $n\geq 1$. If $A=\1_{\cC}$, we simply say that $v$ is $\otimes$-nilpotent.

\begin{definition}\label{detect nilpotence defn}
  An object $A\in \cC$ \tdef{detects nilpotence} if any map
  $v\colon \sI \to \sJ$ in $\cC$ with compact source $\sI\in \cC^{\omega}$ is $\otimes$-nilpotent
  at $A$ if and only if it is $\otimes$-nilpotent. Moreover, we say
  that a map $f\colon A\to B \in \CAlg(\cC)$ \tdef{detects nilpotence}
  if $B\in \Mod_A(\cC)$ detects nilpotence.
\end{definition}

\begin{example}\label{prop: cons detect nilp}
Let $A\in \cC$ such that the functor $-\otimes A\colon \cC \to \cC$
is conservative. Then by~\cite[\href{https://arxiv.org/pdf/2207.09929\#nul.4.18}{Lemma 4.18}]{cnst}
the object $A\in \cC$ detects nilpotence. As consequence,
see~\cite[\href{https://arxiv.org/pdf/2207.09929\#nul.4.21}{Lemma 4.21}]{cnst}, we learn that
 if $v\colon \sI \to \1_{\cC}$ is $\otimes$-nilpotent, the cofiber $\1_{\cC}/v \in \cC$ detects nilpotence.
\end{example}

As we assume that $\1_{\cC}\in \cC^{\omega}$, any nilpotence detecting 
map of commutative algebras
$f\colon A\to B$ has the property that, the induced map of graded rings 
$\pi_{\ast} A\to \pi_{\ast}B$ contains only nilpotent classes in its kernel. 
One may think of the notion of \cref{detect nilpotence defn} as a highly coherent version
of this $\pi_{\ast}$ condition. 
Crucially, nilpotence detecting maps of $\E_{\infty}$-rings are closed under
base change and composition.

\begin{proposition}\label{basechange detect nilpotence}
  Let $f\colon A\to B $, $g\colon B \to C $ and $h\colon A\to D$ be maps in $\CAlg(\cC)$ and
  suppose $f$ and $g$ detect nilpotence. Then the composite $gf\colon A\to C$
  and the base change $f\otimes_{A}h\colon D \to  D \otimes_{A} B$ detect nilpotence.
\end{proposition}
\begin{proof}
  The first claim is~\cite[\href{https://arxiv.org/pdf/2207.09929\#nul.4.30}{Lemma 4.30}]{cnst}
  while the second is~\cite[\href{https://arxiv.org/pdf/2207.09929\#nul.4.24}{Proposition 4.24}]{cnst}.
\end{proof}

We call an $\E_0$-algebra $(f\colon \1_{\cC} \to A) \in \cC_{\1_{\cC}/}$ a \tdef{weak ring} if the map
$A\otimes f \colon A\to A\otimes A$ admits a retraction $\mu \colon A\otimes A \to A$. Note that we do not take
$\mu$ to be part of the datum of $A$, so that being a weak ring is a \emph{property} of an 
$\E_0$-algebra.

\begin{lemma}\label{lem: weak ring detect nilp}
  Let $f\colon \1_{\cC} \to A$ be a weak ring such that, for any map
  $v\colon \sI\to \1_{\cC}$ with $\sI\in \cC^{\omega}$, the composite $f(v)\colon \sI \to A$
  being null implies that $v$ is $\otimes$-nilpotent. Then $A$ detects nilpotence.
\end{lemma}
\begin{proof}
  Let $w\colon \sI\to \sJ$ be a map with $\sI \in \cC^{\omega}$ that is $\otimes$-nilpotent
  at $A$. Since $\cC$ is compactly generated and rigid, to check that the weak ring $A$
  detects nilpotence, we can reduce to the case where $\sJ$ is compact
  by~\cite[\href{https://arxiv.org/pdf/2207.09929\#nul.4.15}{Lemma 4.15}]{cnst} and
  thus dualizable.
  Given a nullhomotopy of the map
  $w^{\otimes n} \otimes A \colon \sI^{\otimes n}\otimes A \to \sJ^{\otimes n}\otimes A$,
  we may pre-compose with the unit $f\colon \1_{\cC}\to A$ and pass to the mate
  to obtain a nullhomotopy of the composite
  \[f(w^{\otimes n})^{\natural}\colon \sI^{\otimes n}\otimes (\sJ^{\otimes n})^{\vee}\to \1_{\cC}\to A\,.\]
  Since taking the mate is symmetric monoidal
  i.e.~$(w^{\otimes n})^{\natural}\simeq (w^{\natural})^{\otimes n}$, the assumption on 
  $f\colon \1_{\cC}\to A$ implies
  that $w^{\natural}\colon \sI\otimes \sJ^{\vee} \to \1_{\cC}$ is $\otimes$-nilpotent. Thus, by the
  same reasoning, $w$ is $\otimes$-nilpotent and we conclude.
\end{proof}

We have the following closure property of filtered colimits of $\E_0$-algebras with nilpotent fiber.

\begin{lemma}\label{filtered colimit detect nilpotence}
  Let $f_i\colon \1_{\cC} \to A_i$ for $i\in I$ be a filtered diagram of
  $\E_0$-algebras in $\cC$ and write $f\colon \1_{\cC} \to A=\fcolim A_{i} \in \cC_{\1_{\cC}/}$
  for the colimit. If each $f_i\colon \1_{\cC}\to A_i$ has $\otimes$-nilpotent fiber and $A$
  is a weak ring, then $A$ detects nilpotence.
\end{lemma}
\begin{proof}
Let $v\colon \sI\to \1_{\cC}$ be a map with $\sI\in \cC^{\omega}$ such that the composite
$f(v)\colon \sI\to \1_{\cC} \to A$ is null. Since $A$ is a weak ring, \cref{lem: weak ring detect nilp}
tells us that it suffices to show that $v$ is $\otimes$-nilpotent.

As $\sI$ is compact, the nullhomotopy of $f(v)$ factors through a finite stage
of the filtered colimit, i.e.~we obtain a nullhomotopy of $f_i(v)\colon \sI\to \1_{\cC} \to A_i$.
for some $i$. In particular, $v$ factors through the fiber of
$f_i\colon \1_{\cC} \to A_i$, which was $\otimes$-nilpotent by assumption, and so $v$ is $\otimes$-nilpotent.
\end{proof}

\subsection{Geometric points over \texorpdfstring{$\F_2$}{F2}}\label{ssec: even points}

We now work with the stable rigid category $\Mod_{\F_2}(\Sp)$ of modules
over the Eilenberg--MacLane spectrum $\F_2$. Given a commutative ring $A\in \CAlg_{\F_2}$,
we want to construct a ring map $A\to T$ such that $\pi_\ast T$ is a graded field.
The procedure is in some sense straightforward,
given a nilpotent element $v\in \pi_\ast A$, we take the $\E_\infty$-cofiber
\begin{equation}\label{eq: cofiber}
 A\too \modEn{A}{\infty}{v} \in \CAlg_{\F_2}
 \end{equation}
to kill $v$. Iterating this and taking the filtered colimit, we arrive at a ring $T^{\prime}$
which has no non-trivial nilpotent elements in $\pi_0T^{\prime}$. 
Since we want to further map to a field however, we need to show that 
$T^{\prime}$ is \emph{not} the zero ring
and this is the difficult part. We ensure this via \cref{prop: killing nilpotence}
by showing that the map~\eqref{eq: cofiber} \emph{detects nilpotence} in the sense of
\cref{detect nilpotence defn}. In particular, this implies that at no stage in
the construction of $T^\prime$ we ended up with the zero ring, and our main result for $p=2$
in the form of \cref{thm: existence of reduced points} follows.

The natural reason why the map~\eqref{eq: cofiber} might \emph{fail} to detect nilpotence,
is that it also kills all the power operations $Q_i(v)$ of \cref{cons: even operations}
and their iterates. However, these obstructions are themselves nilpotent on nilpotent classes,
as the following lemma shows.

\begin{lemma}\label{lem: Qnilpotent}
  Let $A\in \CAlg_{\F_2}$ and let $v\in \pi_{\ast}A$ be some class.
  Then, for any composite operation $Q_{a}^{i_a}\dots Q_{b}^{i_b}$ with $a\leq b$ we have that
  \[Q_{2^{n}a}^{i_a}\cdots Q_{2^{n}b}^{i_b}(v^{2^n})=Q_{a}^{i_a}\cdots Q_{b}^{i_b}(v)^{2^n} \in \pi_\ast A.\] In particular, if $v$ is nilpotent of exponent $m\leq 2^{n}$, then so is each iterated
  operation $Q_I(v)$.
\end{lemma}
\begin{proof}
  By the Cartan formula of \cref{prop: even operations} we know that for any $i\geq 0$
  we have
  \[ Q_i(v^2)= \sum_{k+l=i}Q_k(v)Q_l(v)\,.\]
  Observe that, every decomposition $i=k+l$ appears exactly twice, unless $l=k$
  which only appears if $i$ is even.
  Since $2=0\in \F_2$, this means that $Q_{i}(v^{2})$ vanishes, unless every
  $i$ is divisible by $2$, in which case the only remain term is
  \[ Q_i(v^2)= Q_{i/2}(v)^2.\]
  The claim then follows easily by iterating this argument.
\end{proof}

In \cref{cor: filtered cofiber p=2} we used the description of the free algebra
obtained in \cref{prop: En decomposition F2} to obtain a decomposition of the
$\E_n$-cofiber by some class. We now analyze this decomposition more closely
in the case $n=1$, which provide us with the minimal amount of structure we need
to deduce detecting nilpotence.

\begin{lemma}\label{lem: fil E1 cofib}
  Let $A$ be an $\E_2$-algebra over $\F_2$ and let $v\colon \Sigma^t\F_2\to A$ be
  some class. Then the filtered $\E_1$-cofiber
\[A\too \modEn{A}{1}{v\tau} \in \EAlg{1}{\Mod_{\F_2}^{\fil}}\]
has filtration stage $(2^{i+1}-1)$ given by the $A$-module
\[ A/(v, Q_1(v), \dots, Q_1^{i}(v))\,.\]
\end{lemma}
\begin{proof}
  By \cref{cor: filtered cofiber p=2} we have a filtered colimit description of the
  $\E_1$-cofiber of the form
\[ \modEn{A}{1}{\tau v} \simeq \fcolim_i\left( A/\tau v \otimes_A A/Q_1(\tau v)
\otimes_A \cdots \otimes_A A/Q^i_1(\tau v)\right) \in \Mod_{A}^{\fil} \,.\]
The operation $Q_1$ of \cref{cons: even operations} acting on $\Mod_{\F_2}^{\fil}$ equipped
with the $\tau$-grading has arity 2 and hence $Q_1^i(v\tau) \simeq \tau^{2^{i}}Q^{2^i}_{1}(v)$
is a map of filtered modules
\[ Q_1^{i}(v\tau)\colon \Sigma^{2^i(t+1)-1} A(2^{i}) \too A(0)\,.\]
Thus, the cofiber is given by the filtered $A$-module
\[ \dots \to 0\to A \to \dots \to A \to A/Q^i_1(v) \to A/Q^i_1(v)\to \dots \]
where the first $A/Q^i_1(v)$-term sits in filtration degree $2^{i}$.
Since the monoidal structure on $\Mod_A^{\fil}$ is given by Day-convolution,
we may argue inductively to learn that the filtration on each finite tensor product
\[  A/\tau v \otimes_A A/Q_1(\tau v)
\otimes_A \cdots \otimes_A A/Q^i_1(\tau v)\]
stabilizes at stage $2^{i+1}-1$. Moreover, tensoring with the further cofiber
$A/Q_{1}^{i+1}(\tau v)$ does not change the terms in filtration degree $\leq 2^{i+1}-1$ and hence
the $2^{i+1}-1$ stage of the filtered $\E_{1}$-ring $\modEn{A}{1}{\tau v}$
is given by
\[ \tau^{-1}\left(A/\tau v \otimes_A A/Q_1(\tau v)
    \otimes_A \cdots \otimes_A A/Q^i_1(\tau v)\right)
  \simeq A/(v,Q_{1}(v), \dots, Q_{1}^{i}(v))\]
as claimed.
\end{proof}

We now use the decomposition of \cref{prop: En decomposition F2} to upgrade
\cref{lem: Qnilpotent} to a highly coherent statement about killing nilpotent classes.

\begin{proposition}\label{prop: killing nilpotence} 
  For any $A\in \CAlg_{\F_2}$ and any nilpotent class $v\colon\Sigma^t \F_2\to A$, the
 $\E_\infty$-cofiber $A\to \modEn{A}{\infty}{v}$ detects nilpotence.
\end{proposition}
\begin{proof}
  Since detecting nilpotence is stable under composition and base change
  by \cref{basechange detect nilpotence}, and we can kill $v$ by iteratively
  by killing the smallest non-zero power of $v$ which squares to zero, we
  may assume that $v^2=0$.
  By inverting the filtration parameter $\tau$ in the $\E_1$-decomposition of the 
  $\E_\infty$-cofiber obtained in \cref{cor: filtered cofiber p=2},
  we learn that we have a description of $\modEn{A}{\infty}{v}$ as an infinite tensor product
  of $\E_1$-cofibers
  \begin{equation}\label{eq: E2 final decomp}
   \modEn{A}{\infty}{v} \simeq \bigotimes_{I} (\modEn{A}{1}{Q_I(v)})
   \in \Alg_{\E_1}(\Mod_{A})\,,
   \end{equation}
   where $I$ runs over all sequences of non-negative integers $I=(i_{1},\dots, i_{n})$ of
   varying length $n$ and $Q_I(v)=Q_1^{i_{1}}\cdots Q_n^{i_n}(v)$.
   Since nilpotence detecting $\E_1$-algebras are closed under tensor products
   by~\cite[\href{https://arxiv.org/pdf/2207.09929\#nul.4.14}{Lemma 4.14}]{cnst}
   and filtered colimits by~\cite[\href{https://arxiv.org/pdf/2207.09929\#nul.4.25}{Lemma 4.25}]{cnst},
   it suffices to show that, for each $I$, the $\E_1$-cofiber
   \[ A\too \modEn{A}{1}{Q_I(v)}\]
   detects nilpotence. 
   
   Write $w_0=Q_I(v)$ and $w_i=Q_1^i(w_0)$, then
   the analysis of \cref{lem: fil E1 cofib}, the filtration on this $\E_1$-cofiber
   obtained in \cref{cor: filtered cofiber p=2} admits a cofinal system of the form
   \[ A \to A/w_0 \to A/(w_0,w_1)\to \dots\,.\]
   Moreover, writing $A_i=A/(w_0,\dots, w_i)$ and $\varphi_{i,j}\colon A_i \to A_j$ for
   the structure maps, the $\E_1$-structure on the filtered object
   $\modEn{A}{1}{\tau w_0}$ provides us with multiplication maps $\mu_i$, together with
   a commutative diagram of $A$-linear maps
   \[\begin{tikzcd}
	{A_i \otimes A_i} & {A_{i+1}} \\
	{A_i}
	\arrow["{\mu_i}",from=1-1, to=1-2]
	\arrow["{A_i \otimes \varphi_{0,i}}", from=2-1, to=1-1]
	\arrow["{\varphi_{i,i+1}}"', from=2-1, to=1-2]
\end{tikzcd}\] 
   We know from \cref{lem: Qnilpotent} that $w_i^2=0$ in $A$,
   so the maps $\varphi_{i,i+1}\colon A_i \to A_{i+1}= A_{i}/w_{i+1}$
   and $A\to A/w_j$ have $\otimes$-nilpotent fiber for all $i$ and $j$.
   In particular, each $A/{w_j}$ is a conservative $A$-module and so the $A_i$ are
   conservative as well, as conservative objects are closed under tensor products.
   For any $i$, consider the fiber sequence of $A$-modules
   \[ \sI_{i} \xrightarrow{\alpha_i} A \xrightarrow{\varphi_{0,i}} A_i\,.\]
   Our goal now is to show that the map $\alpha_i$ is $\otimes$-nilpotent
   for each $i$. 
   Tensoring up with $A_i$ we obtain another fiber sequence
   \[ \sI_i \otimes A_i \xrightarrow{\alpha_i \otimes A_i} A_i \to A_i \otimes A_i\,,\]
  which we may compose with the multiplication $\mu_i$, to obtain a nullhomotopy of the composite
   \[ \sI_i \otimes A_i \xrightarrow{\alpha_i \otimes A_i} A_i
   \xrightarrow{\varphi_{i,i+1}} A_{i+1}\,.\]
   In particular, the map $\alpha_i\otimes A_i$ factors through the fiber of $\varphi_{i,i+1}$
   and is thus $\otimes$-nilpotent, meaning there exists some $N\geq 1$ such that
   that $\alpha_i^{\otimes N}\otimes A_i^{\otimes N}\simeq 0$. As the $A$-module
   $A_i$ is conservative, the tensor product $A_i^{\otimes N}$ is conservative as well
   and thus detects nilpotence by \cref{prop: cons detect nilp}. In particular, each 
   $\alpha_i$ is $\otimes$-nilpotent, as claimed.

  Thus, we have shown that the $\E_1$-cofiber $A\to \modEn{A}{1}{Q_I(v)}$ can be written
  as filtered colimit of $\E_0$-algebras $A\to A_i$, each of which has nilpotent fiber.
  Thus, we are in the situation of \cref{filtered colimit detect nilpotence} and learn that
  $A\to \modEn{A}{1}{Q_I(v)}$ detects nilpotence, so we conclude.
\end{proof}

As a straightforward consequence, we now get the following:
\begin{corollary}\label{cor: even prime reduced point}
  Any $A\in \CAlg_{\F_2}$ admits a nilpotence detecting $\E_{\infty}$-map
  $A\to T$ such that $T$ $\pi_*T$ is a reduced ring.
\end{corollary}
\begin{proof}
  Set $A_0:= A$ and denote by $\sI_0\subseteq \pi_*A$ the ideal of all nilpotent
  elements. By \cref{prop: killing nilpotence}, we know that for each $v \in \sI_0$
  the $\E_{\infty}$-map $A_0\to \modEn{A_0}{\infty}{v}$ detects nilpotence.
  In particular, since the $\E_{\infty}$-cofiber
  \[A_0 \too \modEn{A_0}{\infty}{\sI_0} \]
  is a filtered colimit of tensor products of nilpotence detecting $\E_{\infty}$-maps,
  it detects nilpotence by \cref{basechange detect nilpotence} combined
  with~\cite[\href{https://arxiv.org/pdf/2207.09929\#nul.4.25}{Lemma 4.25}]{cnst}.

  Inductively setting
  \[A_i\coloneq (\modEn{A_{i-1}}{\infty}{\sI_{i-1}}) \in \CAlg_{A}\, ,\]
  where $\sI_{i-1}\subseteq \pi_* A_{i-1}$ is the ideal of nilpotent elements, we obtain a
  filtered diagram of $\E_\infty$-maps
  \[ A_0 \to A_1 \to A_2 \to \dots\]
  each of which detects nilpotence. Setting $A_\infty:= (\fcolim A_i)\otimes F$ and again
  using \cref{basechange detect nilpotence}, we deduce that the map
  $A\to A_\infty$ detects nilpotence.
  Since $\F_2\in \Mod_{\F_2}$ is compact, any non-zero
  class $v\colon \Sigma^t \F_2\to A_\infty$ must be detected in some finite stage and hence
  cannot be nilpotent by construction. We conclude that $A_\infty$ is reduced and have thus constructed the desired map.
\end{proof}

Moreover, we now learn the following about the geometric points of $\F_2$-algebras.

\begin{proposition}\label{prop: even prime nstz description}
  Let $A\in \CAlg_{\F_2}$ and let $A\to T$ be a geometric point of $A$.
  Then $T$ is 1-periodic and $\pi_{0}T$ is a separably closed field that is not perfect.
\end{proposition}
\begin{proof}
  Observe that the localization
  \[\F_2 \too \F_{2}\{u^{\pm}\}=\F_{2}\{u\}[u^{-1}]\in \CAlg_{\F_{2}},\]
  where $u$ is an class in topological degree 1, is a conservative $\F_2$-module and
  a compact $\F_2$-algebra. Hence, the base change $T\to T\{u^{\pm}\}$ is non-zero and
  admits a retract, meaning $T$ is 1-periodic with respect to some invertible
  class $u\colon \Sigma\F_2 \to T$.

  Moreover, if $v\in \pi_{0}T$ is nilpotent, then
  by \cref{prop: killing nilpotence} the map $T\to \modEn{T}{\infty}{v}$ detects
  nilpotence. In particular, $T\to \modEn{T}{\infty}{v}$ is non-zero and thus admits a retraction
  i.e.~we have a nullhomotopy $v\simeq 0 \in \pi_{0}T$. Since $T$ is nullstellensatzian, any non-nilpotent
  class in $\pi_0T$ is invertible and so $\pi_{0}T$ is a field. Now, if $f\in \pi_0T [x]$ is
  any separable polynomial, the separable algebra $\pi_0T \to \pi_0T[x]/f=\pi_0 T[f]$ 
  admits a unique lift to a separable $\E_{\infty}$-algebra $T\to T[f]$, for example
  by~\cite[\href{https://arxiv.org/pdf/2305.17236\#thmx.2}{Theorem B}]{ramzisep}.
  In particular map $T\to T[f]$ is non-zero and hence admits a retraction,
  showing that $\pi_0 T$ is separably closed.
  
  Lastly, since the $\F_2$-module $\F_{2}\{x\}[Q_{1}(x)^{-1}]$
  is conservative, it is non-zero after base changing to $T$, and so there exists
  an element $v\in \pi_{0}T$ such that $Q_{1}(v)\neq 0$. However, since $Q_{1}(w^{2})\simeq 0$
  for all $w\in \pi_0T$, for example by \cref{lem: Qnilpotent}, we learn that
  $v$ does not admit a square root, so $\pi_{0}T$ is not perfect.
\end{proof}

In particular, this gives our first main theorem.

\begin{theorem}\label{thm: existence of reduced points}
  Let $A\in \CAlg_{\F_2}$ and $v\in \pi_{\ast} A$ be non-nilpotent. Then there exists a map
  of $\E_\infty$-rings $f\colon A\to T$ such that $\pi_{\ast}T$ is a 1-periodic, separably closed, graded
  field and $f(v)$ is invertible. In particular, if $A$ is non-zero, it admits a reduced point.
\end{theorem}
\begin{proof}
By \cref{cor: detecting classes nstzian}, there exists a geometric point $f\colon A\to T$
such that $f(v)$ is a unit. By \cref{prop: even prime nstz description}, the ring $\pi_{\ast}T$ is a 1-periodic,
separably closed, graded field. The last claim is clear by taking the non-nilpotent element $1\in \pi_0A$ of a non-zero ring $A$.
\end{proof}

One might hope than we can do better than \cref{cor: even prime reduced point} and detect nilpotence via the collection of geometric points.
However, this is easily obstructed by the following:

\begin{corollary}\label{cor: even prime nstzians do not detect nilp}
 The $\F_2\{x\}$-linear map $Q_1(x)\colon \Sigma \F_2\{x\}/x \to \F_2\{x\}/x$
 is not $\otimes$-nilpotent but null at any geometric point $T\in \CAlg_{\F_2\{x\}}$. 
\end{corollary}
\begin{proof}
   The fact that $Q_1(x)$ is not $\otimes$-nilpotent is clear. 
   By \cref{prop: even prime nstz description}, any $\E_\infty$-map to 
   a nullstellensatzian $\F_2\{x\}\to T$ must either invert or kill
   $x$. In the former case, we have that $\F_2\{x\}/x \otimes_{\F_2\{x\}} T \simeq 0$
   and in the later we get $Q_1(x)\otimes T\simeq 0$ and so the
   claim follows.
\end{proof}



\begin{remark}
    Since the main technical input was the decomposition of \cref{cor: filtered cofiber p=2},
    the results of this section generalize in a straightforward manner to $\E_\infty$-algebras
    in any rigid $\F_2$-linear category.
\end{remark}

\subsection{Immortal nilpotent classes over \texorpdfstring{$\F_p$}{Fp} for
\texorpdfstring{$p>2$}{p>2}}\label{ssec: odd points}

Throughout this section, fix a prime $p>2$ and write $\F_p\{x\}\in \CAlg_{\F_p}$ for the free $\E_\infty$-algebra
on a generator $x$.

Our goal is to show that there exist
non-zero $\E_\infty$-rings that admit \emph{no} non-zero maps to a ring $T$ such that
$\pi_0T$ is reduced.
Concretely, we want to find an $A\in \CAlg_{\F_p}$ together with a nilpotent
class $v\in \pi_0A$ and an (iterated) operation $P_I(v)$ as in \cref{cons: odd operations}
that is not nilpotent.
Then, we can invert $P_I(v)$ and obtain a non-zero ring $A[P_I(v)^{-1}]$. As any map
$A\to T$ which kills $v$ must factor through the $\E_\infty$-cofiber by $v$, which in turn kills
the invertible class $P_I(v)$, the map is forced to be zero and we conclude.

We do this by considering the \emph{universal} case, i.e.~by using
the decompositions of \cref{prop: En decomposition Fp} and \cref{cor: unit decomp Fp}
to analyze the ring structure of the free algebra $\modEn{\F_p\{x\}}{\infty}{x^{p}}$
on a generator whose $p$-th power is zero.

We use the following terminology for book-keeping the iterated Dyer--Lashof operations.

\begin{definition}\label{defn: quarks}
  Let $I=(\eps_{n},i_{n},\dots, \eps_{1},i_{1})$ be a Dyer--Lashof sequence that
  is $\infty$-allowable on even classes in the sense of \cref{defn: allowable DL}.
  We say that $I$ is
\begin{enumerate}
\item \tdef{purely bosonic} if $P_I$ preserves parity and $\eps_k=0$ for all $k$.
\item \tdef{mixed bosonic} if $P_I$ preserves parity and contains a Bockstein.
\item \tdef{fermionic} if $P_I$ does not preserve parity.
\end{enumerate}
If $I$ is purely bosonic and $n\geq 0$ is an integer, we also write $nI$ for the purely bosonic sequence
which is entry-wise multiplied by $n$.
\end{definition}

The examples to keep in mind are $P_1P_1$ for purely bosonic, $\beta P_1$ for fermionic and
$\beta P_{1/2}\beta P_1$ for mixed bosonic operations.

\begin{lemma}\label{lem: p-power behaviour}
  Let $A\in \CAlg_{\F_p}$ and let $v\in \pi_0A$ be some class.
  If $I$ is a Dyer--Lashof sequence that is not purely bosonic or not divisible by $p$, then
  $P_I(v^p)\simeq 0$.
  Moreover, for any purely bosonic Dyer--Lashof sequence $J$ we have
  \[P_{pJ}(v^{p})= P_J(v)^{p}.\]
  In particular, if $v$ is nilpotent, then so is $P_{J}(v)$.
\end{lemma}
\begin{proof}
  The Cartan relations of \cref{Odd operations} imply that the total operation
  \[ \varphi\colon \pi_0 A \too \pi_\ast A[t^{\pm}] \otimes \Lambda(z), \quad v\mapsto \sum_{i\geq 0} t^i(P_{i}(v)+\beta P_{i}(v)z)\]
  is a map of graded rings, where $\abs{t}=-2(p-1)$ and $\abs{z}=-1$. Plugging
  in a $p$-th power $v^p$ and using that the $p$-th power is additive on
  homogeneous elements, we see that
  \[ \sum_i P_{i}(v)^p t^{ip} = \varphi(v)^p = \varphi(v^p) = \sum_i t^i (P_{i}(v^p)+ \beta P_{i}(v^p)z).\]
  Comparing coefficients, we learn that $\beta P_{i}(v^p)=0$ for all $i$ and
  $P_{i}(v^p)=0$ unless $p$ divides $i$, in which case
  $P_{i}(v^p)= P_{i/p}(v)^p$. The general claim then follows by replacing with
  $P_{i/p}(v)$ and iterating the argument.
\end{proof}

Thus, the theory immediately diverges from the $p=2$ situation even on the ``algebraic'' level.
With this in hand, we are ready to show that an even operation on a nilpotent even
class need not be nilpotent.

\begin{proposition}\label{prop: bosonic nilpotent}
  Let $I$ be a mixed bosonic Dyer--Lashof sequence.
  Then, for a generator $x$ in topological degree $0$, the class
  \[P_I(x)\in \pi_{\ast}\modEn{\F_{p}\{x\}}{\infty}{x^p}\]
  is not nilpotent.
\end{proposition}
\begin{proof}
  Let $A\in \CAlg_{\F_p}$ take a class in even topological degree $v\in \pi_tA$ and
  consider the $\E_\infty$-cofiber $A\to \modEn{A}{\infty}{v^p}$.
  Inverting the filtration parameter in the decomposition of \cref{cor: unit decomp Fp},
  we learn that we have a $A$-linear decomposition of $\modEn{A}{\infty}{v^{p}}$
  as a filtered colimit of tensor products of $\E_0$-algebras of the form
  \[A\too \modEn{A}{0}{P_I(v^p)}= A/P_I(v^p)\quad \text{and} \quad A\too \modEn{A}{1}{P_J(v^p)}\,,\]
  where the $P_I$ are bosonic and the $P_J$ are fermionic.
  By \cref{lem: p-power behaviour}, we know that $P_K(v^p)\simeq 0$ unless $K$ is purely
  bosonic and each $i_k\in K$ is divisible by $p$, in which case $P_K(v^p)\simeq P_{K/p}(v)^p$.

  Thus, the decomposition is given by a filtered colimit of
  tensor products of $\E_0$-algebras of the form
   \[ A\too \free{A}{0}{a_I},\quad A \too \free{A}{1}{z_J},
     \quad A\too A/P_K(v)^p\,,\]
   where $K$ runs over all purely bosonic Dyer--Lashof sequences.

   Now take $A=\F_p\{v\}$ to be the free $\E_\infty$-algebra on a generator $v$, then
   we know from \cref{lem: Free En odd prime homotopy} that the ring
   $\pi_{\ast}\F_p\{v\}$ is free as a graded commutative $\F_p$-algebra on the 
   $\infty$-allowable operations. In particular, if $I$ is mixed bosonic, no power of 
   $P_I(v)$ is divisible by a purely bosonic $P_K(v)$. Thus, $P_I(v)$ does not act 
   nilpotently on any finite stage of the
   decomposition. Since $A\in \Mod_A$ is compact, nilpotence would be detected at a
   finite stage of the filtered colimit, and so the claim follows.
\end{proof}

This now allows us to construct $\E_{\infty}$-rings which admit no reduced points.

\begin{theorem}\label{thm: odd prime}
  Let $x$ be a class in topological degree 0 consider the ring $A=\modEn{\F_p\{x\}}{\infty}{x^p}$.
  For any mixed bosonic Dyer--Lashof sequence $I$ in the sense of \cref{defn: quarks}, the localization
  $A[P_{I}(x)^{-1}]$ is a non-zero $\E_\infty$-ring which admits no non-zero $\E_\infty$-map
  to any commutative ring $T$ such that $T$ is even or $\pi_{0}T$ is reduced.
\end{theorem}
\begin{proof}
  By \cref{prop: bosonic nilpotent} the element $P_I(x)$ is not nilpotent and hence
  $A\to A[P_I(x)^{-1}]$ is not the zero map, as $\Mod_A$ is rigid.
  However, any map of commutative rings
  $f\colon A[P_I(x)^{-1}]\to B$ with $B$ reduced must take the nilpotent class $x$ to $0$ and hence also
  take $P_I(x)$ to $0$, forcing $f=0$ as claimed. Moreover, since $I$ is mixed bosonic, the composite
  \[P_{I}(x) = \beta^{\epsilon_{i_{1}}}P_{i_{1}}\cdots \beta^{\epsilon_{i_{n}}}P_{i_{n}}\]
  contains at least \emph{two} Bocksteins. Let $k\leq n$ be the largest integer such that
  $\eps_{i_{k}}=1$, meaning that the class
  \[ \gamma(x)=\beta P_{i_{k}}P_{i_{k+1}}\cdots P_{i_{n}}(x)\]
  is necessarily in odd degree. If $B$ is even, the map $f$ must take
  $\gamma(x)$ to $0$ and thus also $P_{I}(x)$, again forcing $f=0$.
\end{proof}

\begin{remark}\label{rem: E2 reduced point}
  In fact, combining \cref{rmk: En cofiber operation} with the description of the 
  free $\E_3$-algebra from \cref{lem: Free En odd prime homotopy}, we learn that
  inverting the mixed bosonic class $\beta P_{\frac{1}{2}}\beta P_{1}(x)$ in $\modEn{\F_{p}\{x\}}{\infty}{x^{p}}$
  produces a non-zero $\E_\infty$-ring which admits no $\E_{2}$-maps to a ring that
  is even or reduced.
\end{remark}

We have identified the mixed bosonic operations $P_I$ an obstruction towards killing nilpotent elements.
However, by \cref{prop: bosonic nilpotent} these obstructions vanish on classes which admit
a $p$-th root.
To utilize this, let us first recall the following well known obstruction
theoretic fact.

\begin{lemma}\label{lem: weak ring odd prime}
    Let $A\in \CAlg_{\F_p}$ and let $v\in \pi_{\ast}A$ be
    a class in even topological degree. Then, the cofiber $A\to A/v$ is a weak ring.
\end{lemma}
\begin{proof}
  Since we are in characteristic $p>2$, the 2-torsion operation $Q_{1}$ induced
  by the $2$-cell of $\mathbb{R}P^\infty$ vanishes.
  In particular, the natural cell structure on the tensor product $A/v \otimes_{A} A/v$
  splits, allowing us to construct a retraction $\mu \colon A/v \otimes_{A} A/v \to A$.
\end{proof}

This lets us show that killing nilpotent elements which admit a $p$-th root detects
nilpotence.

\begin{proposition}\label{prop: p-th root nilpotence}
  Let $A\in \CAlg_{\F_p}$ and suppose we have $v\in \pi_0A$ such that
  $v$ is nilpotent and admits a $p$-th root.
  Then the $\E_{\infty}$-cofiber $A\to \modEn{A}{\infty}{v}$ detects nilpotence.
\end{proposition}
\begin{proof}
  Arguing as in the proof of \cref{prop: killing nilpotence} we can reduce to the case $v^p\simeq 0$.
  By \cref{cor: unit decomp Fp}, we have a description of $\modEn{A}{\infty}{v}$ as an
  $\E_{0}$-algebra in $\Mod_{A}$ as a
  filtered colimit of tensor products of terms of the form
  \[A\too A/P_I(v) \quad \text{and} \quad A\too \modEn{A}{1}{P_J(v)}\,,\]
  where the $P_I$ are bosonic and the $P_J$ are fermionic.
  As the map $A\to \modEn{A}{\infty}{v}$ is $\E_\infty$, combining \cite[\href{https://arxiv.org/pdf/2207.09929\#nul.4.14}{Lemma 4.14}]{cnst}
  with~\cite[\href{https://arxiv.org/pdf/2207.09929\#nul.4.25}{Lemma 4.25}]{cnst}
  tells us that it suffices to prove that each of the terms are nilpotence detecting weak rings.

  Let $w\in \pi_0A$ with $w^p\simeq v$ be a $p$-th root of $v$. Then, if $I$ is mixed bosonic,
  \cref{prop: bosonic nilpotent} implies that $P_I(v)\simeq P_I(w^p)\simeq 0$ and hence we have
  \[A\too A/P_I(v) \simeq \free{A}{0}{\sigma(P_I(v))}= A\oplus \Sigma^{\abs{P_I(v)}+1}A
  \in \EAlg{0}{\Mod_{A}},\]
  which admits an $A$-linear retract and is thus both conservative and a weak
  ring. In particular, it detects nilpotence by \cref{prop: cons detect nilp}.
  Moreover, if $P_I$ is purely bosonic, we know that $P_I(v)$ is nilpotent
  by \cref{prop: bosonic nilpotent}. Thus, the cofiber
  \[A\too A/P_I(v)\]
  has $\otimes$-nilpotent fiber and consequently is a nilpotence detecting
  weak ring by \cref{prop: cons detect nilp} combined with \cref{lem: weak ring odd prime}.

  Finally, if $P_J$ is fermionic, then it contains a Bockstein and so we know by
  \cref{prop: bosonic nilpotent} that
  \[ P_J(v)\simeq P_J(w^p)\simeq 0\,,\]
  meaning we get an equivalence of $\E_0$-algebras
  \[A\too \modEn{A}{1}{P_J(v)}\simeq
  \free{A}{1}{\sigma(P_J(v))}\in \EAlg{0}{\Mod_{R}}.\]
In particular, $\modEn{A}{1}{P_J(v)}$ is a weak ring and admits an $A$-linear retraction,
meaning it is conservative and thus detects nilpotence by \cref{prop: cons detect nilp}.

Thus, all the terms in the filtered colimit decomposition of $A\to \modEn{A}{\infty}{v}$
are nilpotence detecting weak rings and we conclude.
\end{proof}

This allows to say some things about what the homotopy groups 
of nullstellensatzian $\F_p$-algebras look like.

\begin{proposition}\label{prop: odd prime nstz description}
  Let $T\in \CAlg_{\F_p}$ be nullstellensatzian. Then $T$ is $2$-periodic, has non-zero odd
  classes and $\pi_0T$ is not a reduced ring. Moreover, $\pi_{0}T$ is a local ring of dimension
  zero, each $v\in \rm{Nil}(\pi_{0}T)$ satisfies $v^{p}\simeq 0$, but the nilradical
  $\mathrm{Nil}(\pi_0T)$ is not a nilpotent ideal.
\end{proposition}
\begin{proof}
  For a generator $u$ in topological degree $2$, the localization $\F_p\{u^{\pm}\}\in \CAlg_{\F_p}^{\omega}$ is non-zero
  and conservative, i.e.~the map $T\to T\{u^{\pm}\}$ admits a retraction since $T$
  is nullstellensatzian, and so $T$ is 2-periodic.
  Similarly, let $A=(\modEn{\F_p\{x\}}{\infty}{x^p})[\alpha(x)^{-1}]$ where $\alpha(x)= \beta P_{\frac{1}{2}}\beta P_1(x)$ and
  $x$ is a generator in topological degree $0$. As the map $\F_p\to A$ admits an
  $\F_{p}$-linear retraction, $A$ is conservative and so the base change
  $A\otimes T$ is non-zero. Moreover, $A\in \CAlg_{\F_p}^{\omega}$ and so the map
  $T\to A\otimes T$ admits a retraction. Considering the composite
  $A\to A\otimes T \to T $, the claim follows immediately from \cref{thm: odd prime}.

  The fact that $\pi_0T$ is local of dimension zero is clear, since $\Mod_{\F_p}$ is rigid and
  hence every class in $\pi_{\ast}$ of a nullstellensatzian is either invertible or nilpotent.
  Moreover, if $v\in \pi_{0}T$ is nilpotent, \cref{prop: p-th root nilpotence}
  tells us that the map $T\to \modEn{T}{\infty}{v^p}$ detects nilpotence. In particular, since
  $T$ is nullstellensatzian, $v^p= 0$ holds in $\pi_0T$.

  Now let $x_1,\dots, x_n$ be generators in topological degree zero and consider the free
  algebra
  \[ B_n=\modEn{\F_p\{x_1,\dots, x_n\}}{\infty}{(x_1^p,\dots, x_n^p)}
  \simeq \modEn{\F_p\{x_1\}}{\infty}{x_1} \otimes \dots \otimes \modEn{\F_p\{x_n\}}{\infty}{x_n}\,.\]
Let $P_I=\beta P_{i_{2n}} \beta P_{i_{2n-1}} \cdots \beta P_{i_1}$ be a mixed bosonic
Dyer--Lashof operation of length $2n$ where every extended power is followed
by a Bockstein.
Expanding the Cartan formula for $P_I(x_1\cdots x_n)$ we see that it contains a non-trivial
term of the form $P_{J_1}(x_1)\cdots P_{J_n}(x_n)$ where the $k$-th factor is
a mixed bosonic operation
\[P_{J_k}(x_k)= P_{j_{k,2n}}P_{j_{k,2n-1}}\cdots \beta P_{j_{k,2k}}\beta P_{j_{k,2k-1}}
\cdots P_{j_{k,2}} P_{j_{k,1}}(x_k)\]
with exactly two Bocksteins. By \cref{lem: Free En odd prime homotopy} combined with the
fact that, since $\F_p$ is a field, taking homotopy groups
is symmetric monoidal, we know that $\pi_{\ast}\F_p\{x_1,\dots, x_n\}$ is a free graded
commutative $\F_p$-algebra on the $\infty$-allowable operations $P_{K}(x_i)$.
Moreover, by the analysis of the $\E_\infty$-cofiber carried out in \cref{prop: bosonic nilpotent},
we know that each $P_{J_k}(x_i)$ does not act nilpotently on the ring $B_n$, which implies
that $P_I(x_1\cdots x_n)$ does not act nilpotently either.

Arguing as in the proof of \cref{thm: odd prime}, we can thus find nilpotent elements 
$v_{1},\dots, v_{n}\in \pi_{0}T$ such that $P_{I}(v_{1}\cdots v_{n})$ is invertible and
hence the product $v_{1}\cdots v_{n}\in \pi_{0}T$ is non-zero. Since $n$ was arbitrary, this proves the claim.
\end{proof}

Lastly, the description of the nullstellensatzian $\F_p$-algebras
of \cref{prop: odd prime nstz description} allows us to deduce
that they do not jointly detect nilpotence.

\begin{corollary}\label{cor: odd prime nstzians do not detect nilp}
    The $\F_p\{x\}$-linear map $P_1(x)\colon \Sigma^{2(p-1)}\F_p\{x\}/x\to \F_p\{x\}/x$ is not itself $\otimes$-nilpotent, but is 
    $\otimes$-nilpotent at any geometric point $T\in \CAlg_{\F_p\{x\}}$.
\end{corollary}
\begin{proof}
    The fact that $P_1(x)$ is not $\otimes$-nilpotent is straightforward.
    Moreover, by \cref{prop: odd prime nstz description}, any $\E_\infty$ map to a nullstellensatzian
    $\F_p\{x\}\to T$ must either invert $x$ or kill $x^p$. Since the 
    latter forces
    $P_1(x)$ to be nilpotent by \cref{lem: p-power behaviour}, the claim
    follows.
\end{proof}


\printbibliography{}

\end{document}